\documentclass[11pt]{amsart}
\usepackage{amssymb, latexsym, amsmath, amsfonts, tikz}
\usepackage{hyperref, enumitem, fancyhdr}
\usepackage{color}
\usepackage{tikz-cd} 

\newtheorem{thm}{Theorem}[section]

\newtheorem{lem}[thm]{Lemma}
\newtheorem{prop}[thm]{Proposition}
\theoremstyle{definition}

\theoremstyle{remark}

\numberwithin{equation}{section}
\theoremstyle{remark}

\newcommand{\no}{\noindent}

\begin{document}
\title{Escaping Sets of Skew Products of H\'{e}non maps} 
\keywords{H\'{e}non maps, Escaping sets}
\author{Mahima}

\address{Indian Institute of Science Education and Research Mohali, Knowledge City, Sector -81, Mohali, Punjab-140306, India}
\email{ph22080@iisermohali.ac.in, mahimamath06@gmail.com}

\begin{abstract}
We study skew products of H\'enon maps fibered over suitable metric spaces and investigate the analytic structure of their escaping sets. Our study extends the description of escaping sets for H\'enon maps developed by Hubbard and Oberste-Vorth \cite{HOV} to skew products and provides a framework for studying their rigidity. For a compact base space, we construct an intermediate covering space of each fiberwise escaping set. When the base is the closed unit disk and the family depends holomorphically on the parameter, we obtain an analogous description of the global escaping set. We further use this covering space construction to study the relationship between biholomorphic equivalences of their global escaping sets and the underlying dynamics. Finally, for skew products over the non-compact base $\mathbb{C}$, we consider the corresponding escaping region and investigate its analytic structure.

\end{abstract}

\maketitle
\large

\section{Introduction}
In this article, we study skew products of H\'enon maps fibered over suitable metric spaces. These are maps of the form
$H : M \times \mathbb{C}^2 \longrightarrow M \times \mathbb{C}^2$
defined by
\begin{equation}\label{skew}
H(\lambda,x,y) = (\sigma(\lambda), H_\lambda(x,y)),
\end{equation}
where $\sigma$ is a continuous self-map of $M$ and, for each $\lambda \in M$, $H_\lambda$ is a H\'enon map given by
\[
H_\lambda(x,y) = (y, q_\lambda(y) - \delta_\lambda x).
\]
Here
\[
q_\lambda(y) = y^d + c_{\lambda,d-2} y^{d-2} + c_{\lambda,d-3}y^{d-3}+\cdots + c_{\lambda,0},
\]
is a monic and centered polynomial of degree $d \geq 2$, whose coefficients $c_{\lambda,i} = c_i(\lambda)$ for $0 \le i \le d-2$, and the Jacobian determinant, $\delta_\lambda = \delta(\lambda)$ depend continuously on $\lambda$. We consider this normalized form of $q_\lambda$ since any skew product
\[
(\lambda,x,y) \mapsto \left(\sigma(\lambda), y, \tilde c_d(\lambda)y^d + \tilde c_{d-1}(\lambda)y^{d-1} + \cdots + \tilde c_0(\lambda) - \tilde\delta(\lambda)x\right),
\]
can be conjugated by a map of the form
\[
(\lambda,x,y) \mapsto \left(\lambda, \tilde a(\lambda)x + \tilde b(\lambda), \tilde a(\lambda)y + \tilde b(\lambda)\right),
\]
so that the resulting map is of the form \eqref{skew}. 

\medskip
\no
H\'enon maps are dynamically interesting objects and are widely studied.  Up to conjugacy, Friedland and Milnor in \cite{FM} classified polynomial automorphisms of $\mathbb{C}^2$ into three classes: affine maps, elementary maps, and finite compositions of H\'enon maps. The dynamics of the first two classes, that is, affine maps and elementary maps, is very simple. On the other hand, although H\'enon maps may appear to be simple but it turns out that they exhibit chaotic behavior under iteration. Their dynamical properties have been studied in pioneering works of Bedford-Smillie, Forn\ae ss-Sibony and Hubbard-Oberste Vorth (see \cite{BS1}, \cite{FS}, \cite{HOV}, \cite{HOV2}).

\medskip
\no
The skew products of H\'enon maps considered in this paper come naturally in dynamics. These maps arise in dynamics of higher dimensional polynomial automorphisms. Although not much is known about the dynamics of polynomial automorphisms of $\mathbb{C}^N$, $N\geq 3$, the class of quadratic polynomial automorphisms of $\mathbb{C}^3$ was studied by  Fornaess-Wu in \cite{FW}. They classified them into seven classes. One of these classes is given by
\[
H_1(x,y,z) = \bigl(cx+d,\;Q(x)+z,\; P(x,z)+a y\bigr),
\]
where $P$ and $Q$ are polynomials with $\max\{\deg P, \deg Q\} = 2,$
and $a c \neq 0$.
If we assume that $Q(x)=0$ and define $\sigma(x)=c x+d$, then
\[
H_1(x,y,z)
= \bigl(\sigma(x),\; z,\; P(x,z)+a y\bigr),
\]
which is precisely a skew product of the form \eqref{skew}. Thus, the class studied in this article naturally includes a subclass of polynomial automorphisms of $\mathbb{C}^3$. Furthermore, skew products of the form \eqref{skew} also provide examples of {\it regular} automorphisms (see \cite{S99}).

\medskip
\noindent
One of the fundamental dynamical objects in the study of H\'enon maps is the set of points that escape to infinity under forward iterates, called the {\it escaping set}. This set is exactly the set of points where the Green's function is positive. Hubbard and Oberste-Vorth in \cite{HOV} gave the analytic description of the escaping set and constructed an intermediate cover of this set which is biholomorphic to $\mathbb{C}\times\mathbb{C}\backslash\mathbb{\overline{D}}$, here $\mathbb{D}$ is the unit disk. The motivation behind considering this particular covering space was to analytically extend the B\"{o}ttcher's function (also see \cite{MNTU}). This intermediate cover turns out to be an important tool to study several rigidity results for H\'enon maps (see \cite{BRT},\cite{Pal}). 

\medskip
\no 
The dynamical properties of the skew products of H\'enon maps fibered over a compact set were studied in \cite{PV}. For every $\lambda\in M$ and every $n\geq 1$, define
\[
H_\lambda^n=H_{\sigma^{n-1}(\lambda)}\circ H_{\sigma^{n-2}(\lambda)}\circ\cdots\circ H_\lambda.
\]
In the compact case, for every $\lambda\in M$, the sets of {\it escaping} and {\it non-escaping} points are defined as
\[
I_\lambda^+ =
\left\{ (x,y) \in \mathbb{C}^2 :\|H_\lambda^n(x,y)\| 
\to \infty \text{ as } n \to \infty \right\},
\]
and
\[
K_\lambda^+ =
\left\{ (x,y) \in \mathbb{C}^2 :
\left\{ H_\lambda^n(x,y)\right\}_{n \ge 0}
\text{ is bounded} \right\},
\]
respectively. It turns out that $H_\lambda(I_\lambda^+)=I_{\sigma(\lambda)}^+$, $H_\lambda(K_\lambda^+)=K_{\sigma(\lambda)}^+$ and $I^+_\lambda=\mathbb{C}^2\backslash K_\lambda^+$. Motivated by the filtration properties of H\'enon maps (see \cite{BS1}), Pal-Verma in \cite{PV} proved the existence of a uniform filtration radius $R>0$, independent of $\lambda$, for which the map \eqref{skew} possesses similar filtration properties. For $R>0$, define 
\[
V_R =\{(x,y) \in \mathbb{C}^2 :\max\{ \vert x\vert,  \vert y \vert\}\leq R \},
\] 
\[
V_R^+=\{(x,y) \in \mathbb{C}^2 :\max\{ \vert x \vert,R\} < \vert y \vert \},
\]
and
\[
V_R^-=\{(x,y) \in \mathbb{C}^2 : \max\{\vert y \vert,R\} < \vert x \vert \}.\]
Then, for each $\lambda\in M$, 
\[
I_\lambda^+=\bigcup_{n\geq 0}(H_\lambda^n)^{-1}(V_R^+),
\]
where $\{(H_\lambda^n)^{-1}(V_R^+)\}_{n\geq 0}$ is an increasing sequence of open sets (see \cite[Lemma 2.1]{PV}). 

\medskip
\no
The main objective of this paper is to extend the analytic description of the escaping set given by Hubbard and Oberste-Vorth to the skew product setting. For each $\lambda\in M$, we construct a covering space for the escaping set $I_\lambda^+$.\\
Moreover, we consider the special case where $M=\overline{\mathbb{D}}$. Assume that $\sigma:\overline{\mathbb{D}}\to\overline{\mathbb{D}}$ is a continuous map whose restriction to $\mathbb{D}$ is an automorphism of $\mathbb{D}$. We also assume that the coefficients $c_{\lambda,i}$, for $0 \le i \le d-2$, of the polynomial $q_\lambda$, together with $\delta_\lambda$, depend holomorphically on $\mathbb{D}$ and extend continuously to $\overline{\mathbb{D}}$. Under these assumptions, we define the {\it global escaping set} by
\[
I^+ =
\left\{ (\lambda,x,y) :
\lambda \in \mathbb{D},\ (x,y) \in I_\lambda^+ \right\}.
\]
Since $H_\lambda(I_\lambda^+)=I^+_{\sigma(\lambda)}$, it follows immediately that $H(I^+)=I^+$. We study the analytic structure of the global escaping set.
\begin{thm}\label{tmh1}
For each $\lambda \in M$, we have the following
\begin{itemize}
    \item[(i)] the fundamental group of $I_\lambda^+$ is isomorphic to $\mathbb{Z}\left[1/d\right]$,
    \item[(ii)] the covering manifold $\hat{I}_\lambda^+$ of $I_\lambda^+$ corresponding to the subgroup $\mathbb{Z}$ is biholomorphic to $\mathbb{C}\times\mathbb{C}\backslash\mathbb{\overline{D}}$.
\end{itemize}  
Furthermore, if $M=\mathbb{\overline{D}}$, then
\begin{itemize}
    \item[(iii)] the fundamental group of the global escaping set $I^+$ is isomorphic to $\mathbb{Z}\left[1/d\right]$,
    \item[(iv)] the analytic covering manifold $\hat{I}^+$ with fundamental group isomorphic to $\mathbb{Z}$ is biholomorphic to $\mathbb{D}\times\mathbb{C}\times\mathbb{C}\backslash\mathbb{\overline{D}}$,
    \item[(v)] the skew product $H$ lifts to a holomorphic map $\tilde{H}:\mathbb{D} \times \mathbb{C} \times \mathbb{C}\backslash\mathbb{\overline{D}}\to \mathbb{D} \times \mathbb{C} \times \mathbb{C}\backslash\mathbb{\overline{D}}$ given by 
\[
\tilde{H}(\lambda,s,t) = \left(\sigma(\lambda),\left(\frac{\delta_\lambda}{d}\right)s+Q(\lambda,t),t^d\right),
\]
where, for each $\lambda\in M$, $Q_\lambda(t)=Q(\lambda,t)$ is a polynomial in $t$ of degree $d+1$.
\end{itemize}
\end{thm}

\medskip
\no
The existence of an intermediate cover has important applications. For classical H\'enon map, Bonnot-Radu-Tanase in \cite{BRT} investigated the question---\textit{To what extent does the escaping set of H\'{e}non maps determine the underlying H\'{e}non map?} Using the techniques of Bousch \cite{Bousch} and the intermediate cover constructed in \cite{HOV}, they proved that two H\'{e}non maps $F_1$ and $F_2$ of degree $2$ are equal if and only if their escaping sets are biholomorphic. Further, Pal in \cite{Pal} generalized this result and proved the rigidity of the escaping sets of H\'enon maps of degree $\geq 2$. Two H\'{e}non maps $F_1$ and $F_2$ of same degree and with biholomorphic escaping sets are equal up to composition with linear maps i.e., $F_1= L_1\circ F_2\circ L_2$, where $L_1$ and $L_2$ are linear automorphisms on $\mathbb{C}^2$. Since we have constructed an intermediate covering space for the global escaping set, it is natural to ask whether a similar rigidity result holds in the skew product setting. 
 
\begin{thm}\label{thm2}
Let $H,F:\mathbb{\overline{D}}\times\mathbb{C}^2\to \mathbb{\overline{D}}\times\mathbb{C}^2$ defined by $H:(\lambda,x,y) \mapsto (\lambda,y,q_{H,\lambda}(y)-\delta x)$ and $F:(\lambda,x,y) \mapsto (\lambda,y,q_{F,\lambda}(y)-\delta x),$ where \[q_{H,\lambda}(y)=y^d+a_{H,\lambda,d-2}y^{d-2}+\cdots +a_{H,\lambda,0}\] and \[q_{F,\lambda}(y)=y^d+a_{F,\lambda,d-2}y^{d-2}+\cdots +a_{F,\lambda,0},\] are two families of skew products of H\'{e}non maps. Suppose there exists a biholomorphism between the global escaping sets of the two families which lifts to an automorphism on the covering space $\mathbb{D}\times\mathbb{C}\times\mathbb{C}\backslash\mathbb{\overline{D}}$, then
\[
F= L_1 \circ H \circ L_2
\]
where $L_1(\lambda,x,y)=(A_1(\lambda),\alpha \alpha_1 x,\alpha^d y)$, and $L_2(\lambda,x,y)=(A_1^{-1}(\lambda),\alpha^{-d} x,\alpha^{-1} y)$, with $A_1:\mathbb{D}\to \mathbb{D}$ is a biholomorphism, $\alpha^{d^2-1}=1$ and $\alpha^{d-1}_1=1$.
\end{thm}

\medskip
\no
In the final part of the paper, we consider the skew products over non-compact base space $M=\mathbb{C}$. 
Let $H:\mathbb{C}\times \mathbb{C}^2 \to \mathbb{C}\times \mathbb{C}^2$ defined by 
\begin{equation}\label{skewhenon}
    H(\lambda,x,y)=(c\lambda,H_\lambda(x,y)),
\end{equation}
where $H_\lambda(x,y)=(y,p_\lambda(y)-\delta x)$ is a H\'enon map of degree $d\geq 2$ with $p_\lambda(y)=c_d y^d +c_{\lambda,d-1} y^{d-1}+ \cdots +c_{\lambda,0}$. The coefficients $c_{\lambda,j}=c_j(\lambda)$, for $0 \leq j\leq d-1$, are polynomials in $\lambda$.
The inverse of $H$ is given by
\[
H^{-1}(\lambda,x,y)=\left(c^{-1}\lambda,H^{-1}_{c^{-1}\lambda}(x,y)\right)=\left(c^{-1}\lambda,\left(p_{c^{-1}\lambda}(x)-y\right)/\delta,x\right).
\]
For $0 \leq j \leq d$, let $l_j=\deg(c_j)$ and $\deg(H) = \tilde{d}= \max \{l_j +j : 0\leq j \leq d \}.$ Then for $n\geq 2$, $\deg(H^n)=\tilde{d}d^{n-1}$. The dynamics of the maps of the form \eqref{skewhenon} is studied in \cite{BP}. For $|c|>1$ and $R>0$, we define
\[
V_R^+ = \{(\lambda,x,y) \in \mathbb{C}^3 : |y| > \max \{R, |x|, |\lambda|^{\tilde{d}+1} \}\},
\]
and
\[
V_R^-= \{(\lambda,x,y) \in \mathbb{C}^3 : |x| > \max \{R, |y|\}, |\lambda|<1 \}.
\]
Define
\[
U_H^+ = \bigcup_{n\geq 0} H^{-n}(V^+_R).
\]
For quadratic polynomial automorphisms of $\mathbb{C}^3$, $U_H^+$ is studied by Coman-Forn\ae ss in \cite{CF}. They observed that for sufficiently large $R>0$, $U_H^+$ contains points that escape to infinity at super-exponential rate under forward iteration.  Unlike the case of H\'enon maps, where there are no points that escape to infinity at exponential rate, they proved that the set $\mathbb{C}^3\backslash U_H^+$ not only contains points with bounded forward iterates but also points whose forward iterates escape to infinity at a slower rate (exponential rate). Moreover, $U_H^+$ is precisely the set where the Green's function is positive (see \cite{CF}, \cite{BP}). This makes $U_H^+$ a natural analogue of the classical escaping set. Now we record the following theorem.

\begin{thm}\label{thm3}
    The fundamental group of $U_H^+$ is isomorphic to $\mathbb{Z}\left[1/d\right]$ and the intermediate cover, $\hat{U}_H^+$, of $U_H^+$ whose fundamental group is isomorphic to $\mathbb{Z}$ is biholomorphic to $\mathbb{C}^2\times\mathbb{C}\backslash\mathbb{\overline{D}}$. Moreover, $H$ lifts to the map $\tilde{H}:\mathbb{C}^2\times\mathbb{C}\backslash\mathbb{\overline{D}}:\mathbb{C}^2\times\mathbb{C}\backslash\mathbb{\overline{D}}$ given by
    \[
    \tilde{H}(\lambda,s,t)=\left(c\lambda, \left(\frac{\delta}{d}\right)s+Q_H(\lambda,t),t^d\right),
    \]
    where, for each $\lambda\in M$, $Q_H$ is a polynomial in $t$ of degree $d+1$.
\end{thm}

\medskip
\noindent
\textbf{Acknowledgment}: I would like to thank Ratna Pal for suggesting the problem statement and for her encouragement and support in completing this paper. I also gratefully acknowledge the University Grants Commission (UGC), India, for funding my research.

\section{Proof of Theorem \ref{tmh1}}
{\it Step 1.} In this step, for each $\lambda\in M$, we construct the B\"{o}ttcher function.
For $1\leq i\leq N$, let $proj_i:\mathbb{C}^N\to \mathbb{C}$ be the projection map defined by $proj_i(x_1,x_2,\cdots,x_N)=x_i$. For $(x,y) \in V_R^+$, 
\begin{equation}\label{eq2.1f}
    \frac{proj_2 \circ H_\lambda(x,y)}{y^d}=1+\frac{c_{\lambda,d-2}}{y^2}+\cdots+\frac{c_{\lambda,0}-\delta_\lambda x}{y^d}.
\end{equation}
For sufficiently large $R>0$, the quantity in \eqref{eq2.1f} lies in a simply connected neighborhood of $1$ in $\mathbb{C}^*$ for every $(x,y)\in V_R^+$. Thus, we could choose a branch of logarithm, $\beta_\lambda(x,y)$, on $V_R^+$, such that
\begin{equation}\label{eq2.2f}
proj_2 \circ H_\lambda(x,y)=y^d e^{\beta_\lambda(x,y)}.
\end{equation}
Since for $(x,y)\in V_R^+$, we have $H^{n-1}_\lambda(x,y) \in V_R^+$ for every $n\geq 1$, 
\begin{equation}
proj_2 \circ  H_{\sigma^{n-1}(\lambda)} (H^{n-1}_\lambda(x,y))=\left(proj_2\circ H^{n-1}_\lambda(x,y)\right)^d e^{\beta_{\sigma^{n-1}(\lambda)}(H^{n-1}_\lambda(x,y))}.
\end{equation}
Inductively, we get
\begin{equation}
proj_2 \circ H^n_\lambda(x,y)=y^{d^n}e^{\beta_{\sigma^{n-1}(\lambda)}(H^{n-1}_\lambda(x,y))+d\beta_{\sigma^{n-2}(\lambda)}(H^{n-2}_\lambda(x,y))+...+d^{n-1}\beta_\lambda(x,y)}.
\end{equation}
Let 
\[
\phi_{\lambda,n}(x,y)=\left(proj_2 \circ H_\lambda^n(x,y)\right)^{1/d^n}.
\]
Since M is compact, there exists a uniform bound for $\beta_\lambda$ (independent of $\lambda \in M$) on $V_R^+$. If necessary, we choose a bigger $R$ to ensure this bound holds. Thus, the sequence of holomorphic functions $\gamma_{\lambda,n}:V_R^+\to \mathbb{C}$ defined by
\begin{equation}
    \gamma_{\lambda,n}(x,y)=\frac{1}{d}\beta_\lambda(x,y)+\frac{1}{d^2}\beta_{\sigma(\lambda)}(H_\lambda(x,y))+...+\frac{1}{d^n}\beta_{\sigma^{n-1}(\lambda)}(H_\lambda^{n-1}(x,y)),
\end{equation}
converges uniformly on compact subsets of $V_R^+$ to a holomorphic function $\gamma_\lambda$. Therefore, the sequence of holomorphic functions $\{\phi_{\lambda,n}\}_{n=0}^\infty$ converges uniformly on compact subsets of $V_R^+$ to a holomorphic function $\phi_\lambda$, that is,
\begin{equation}
    \phi_\lambda(x,y)=\lim_{n \to \infty}\phi_{\lambda,n}(x,y)=ye^{\gamma_\lambda(x,y)}.
\end{equation}
\noindent
Since $\phi_{\sigma(\lambda),n}(H_\lambda(x,y))=\phi_{\lambda,n+1}(x,y)^d$, $\phi_\lambda$ satisfies the functional equation
\begin{equation*}
    \phi_{\sigma(\lambda)}(H_\lambda(x,y))=\phi_\lambda(x,y)^d.
\end{equation*}
Inductively, for $n\geq 1$, 
\begin{equation}\label{eqphi}
    \phi_{\sigma^n(\lambda)}(H_\lambda^n(x,y))=\phi_\lambda(x,y)^{d^n}.
\end{equation}

\medskip
\noindent
It follows from \eqref{eq2.1f} and \eqref{eq2.2f} that $e^{\beta_\lambda(x,y)} \approx 1$ as $y \to \infty$. Therefore, $\phi_\lambda(x,y) \approx y$ as $\|(x,y)\|\to \infty$ in $V_R^+$. 
Hence, we could choose large $R$ such that there exist constants $m_1,m_2>0$ such that 
\[
m_1 \leq \left|\frac{\phi_\lambda(x,y)}{y}\right| \leq M_2
\]
for $(x,y) \in V_R^+$.

\medskip
\no
{\it Step 2.} Using the techniques developed in \cite{MNTU}, in this step, we prove that for each $\lambda\in M$ the fundamental group of $I^+_\lambda$ is isomorphic to the group $\mathbb{Z}[1/d]=\left\{k/d^n :k\in \mathbb{Z}, n\in \mathbb{N}\cup\{0\}\right\}$.
Define the closed one-form $\omega_\lambda = \mathrm{d} \phi_\lambda/\phi_\lambda$ on $V_R^+$. Extend it to $I_\lambda^+$ using the functional equation \eqref{eqphi}. For $n\geq 0$, $H_\lambda^n: (H_\lambda^n)^{-1}(V_R^+) \to V_R^+$ is a homeomorphism that pullback closed forms to closed forms, so we define 
\[
\omega_\lambda = \frac{1}{d^n}(H_\lambda^n)^* \omega_{\sigma^n(\lambda)} \quad\text{on $(H_\lambda^n)^{-1}(V_R^+)$}.
\]
On $V_R^+$, 
\[
H_\lambda^* \omega_{\sigma(\lambda)} = H_\lambda^*\left(\frac{\mathrm{d} \phi_{\sigma(\lambda)}}{\phi_{\sigma(\lambda)}}\right)=\frac{\mathrm{d}(\phi_{\sigma(\lambda)}  \circ H_\lambda)}{\phi_{\sigma(\lambda)} \circ H_\lambda} = \frac{\mathrm{d} \phi_\lambda ^d}{\phi_\lambda ^d}=d\frac{\mathrm{d} \phi_\lambda}{\phi_\lambda}=d\omega_\lambda.
\]
Note that $H_\lambda^*\omega_{\sigma(\lambda)}=d \omega_\lambda$ on $I_\lambda^+$.
First, we see that the fundamental group of $V_R^+$ is $\mathbb{Z}$. Let $C$ be a closed path in $V_R^+$, then
\[ \int_C \omega_\lambda = \int_C \frac{d \phi_\lambda}{\phi_\lambda} =\int_C \frac{d(ye^{\gamma_\lambda(x,y)})}{ye^{\gamma_\lambda(x,y)}}=\int_C \frac{dy}{y}+\int_Cd \gamma =\int_C \frac{dy}{y}, \]
which implies
\[ 
\frac{1}{2\pi i}\int_C \omega_\lambda =\frac{1}{2 \pi i}\int_C \frac{dy}{y} \in \mathbb{Z}.
\]
The path $C$ is null-homotopic ($C \sim 0$) if and only if $\int_C \omega_\lambda =0$ and for every $m \in \mathbb{Z}$, $\frac{1}{2\pi i} \int_C \omega_\lambda =m$ if and only if $C \sim \sigma_m$, where for $\theta\in [0,1]$, $\sigma_m:\theta \mapsto (0,2Re^{2 \pi i m \theta})$ is a closed curve in $V_R^+$. Thus, $\pi_1(V_R^+) \cong \mathbb{Z}$.
Now, let $C$ be a curve in $I_\lambda^+$. Since $C$ is compact, there exist $n >0$, such that $ H_\lambda^n(C) \subset V_R^+$, and 
\[
\int_C \omega_\lambda = \frac{1}{d^n}\int_C (H_\lambda^n)^* \omega_{\sigma^n(\lambda)} =\frac{1}{d^n} \int_{H_\lambda^n(C)} \omega_{\sigma^n(\lambda)} \in \mathbb{Z}\left[\frac{1}{d}\right].
\]
Further, 
\[
\int_C \omega_\lambda=0 \Leftrightarrow \int_{H_\lambda^n(C)} \omega_{\sigma^n(\lambda)} =0 \Leftrightarrow H_\lambda^n(C) \sim 0 \Leftrightarrow C \sim 0,
\]
where the last implication is because $(H_\lambda^n)^{-1}: I_{\sigma^n(\lambda)}^+ \to I_\lambda^+$ is a homeomorphism. If we define 
\[
f_\lambda:\pi _1 (I_\lambda^+) \to \mathbb{Z}\left[\frac{1}{d}\right] \text{ by } f_\lambda([C])=\frac{1}{2 \pi i} \int_C \omega_\lambda,
\]
then $f_\lambda$ is an injective group homomorphism. Furthermore for $n_1/d^{n_2}\in \mathbb{Z}\left[1/d\right]$, $f_\lambda\left((H_\lambda^{n_2})^{-1}(\sigma_{n_1})\right)=n_1/d^{n_2}$. Thus, $f_\lambda$ is a group isomorphism. If $C \subset I_\lambda^+$ is a closed curve, then 
\[
f_{\sigma(\lambda)}([H_\lambda(C)])= \frac{1}{2 \pi i} \int_{H_\lambda(C)} \omega_{\sigma(\lambda)} = \frac{1}{2\pi i}\int_C H_\lambda^* \omega_{\sigma(\lambda)}=\frac{1}{2\pi i}\int_C d \omega_\lambda=d f_\lambda([C]).
\]

\medskip
\no
{\it Step 3.}
Now we construct a covering manifold $\hat I_\lambda^+$ of $I_\lambda^+$ with fundamental group $\mathbb{Z}\leq \mathbb{Z}[1/d]$. Fix a point $\alpha \in V_R^+$. Define an equivalence relation $\sim _{\lambda}$ on the set $A_\lambda=\{(z,C): z \in I_\lambda^+$ and C is a path from $\alpha$ to $z$ in $I_\lambda^+ \}$ by

\begin{center}
    $(z,C) \sim_{\lambda} (z',C') \Leftrightarrow z=z'$ and $f_\lambda([C \bar{C}']) \in \mathbb{Z}$.
\end{center}
Set, $\hat I_\lambda^+ =A_\lambda/ \sim_{\lambda}$, the set of equivalence classes of $\sim_{\lambda}$. Define $\pi_\lambda:\hat I_\lambda^+ \to I_\lambda^+$ by
\[
[z,C] \mapsto z.
\]
Give $\hat I_\lambda^+$ a topology such that $\pi_\lambda$ is a continuous map. In fact, by the standard construction of covering spaces $\pi_\lambda$ is a covering, and $\pi_1(\hat{I}_\lambda^+) \cong \mathbb{Z}$. In addition, we can give $\hat{I}_\lambda^+$ a unique complex manifold structure which makes $\pi_\lambda$ a holomorphic map.
Let $\hat V_R^+ = \{[(z,C)] \in \hat{I}_\lambda^+:$ C is a path in $ V_R^+ \}$, then $\pi_\lambda|_{\hat V_R^+} : \hat{V}_R^+ \to V_R^+$ is a biholomorphism. Define $\hat{\phi}_\lambda: \hat{I}_\lambda^+\to \mathbb{C}$ by \[
\hat{\phi}_\lambda([z,C])=\phi_\lambda(\alpha) e^{\int_C \omega_\lambda}.
\]

\medskip
\no
We now establish some properties through some propositions.
\begin{prop} If $z=z'$ and $\hat{\phi}_\lambda([z,C])=\hat{\phi}_\lambda([z',C'])$ then $[z,C]=[z',C']$.
\end{prop}
\begin{proof} Since $\hat{\phi}_\lambda([z,C])=\hat{\phi}_\lambda([z',C'])$, 
\[
\phi_\lambda(\alpha)e^{\int_C \omega_\lambda}=\phi_\lambda (\alpha)e^{\int_{C'} \omega_\lambda},
\]
which implies $f_\lambda([C\bar{C}']) \in \mathbb{Z}$. Further since $z=z'$, $[z,C]=[z',C'].$
\end{proof}

\begin{prop} If $\hat{z}=[z,C] \in \hat{V}_R^+$, then $\hat{\phi}_\lambda(\hat{z})=\phi_\lambda(\pi_\lambda(\hat{z}))$.
\end{prop}
\begin{proof} We have 
\[
\hat{\phi}_\lambda(\hat{z}) = \phi_\lambda(\alpha)e^{\int_C \omega_\lambda} =\phi_\lambda(\alpha)e^{\int_C \frac{d\phi_\lambda}{\phi_\lambda}}= \phi_\lambda(z).\]
\end{proof}

\noindent
Fix a path $L$, from $\alpha$ to $H_\lambda(\alpha)$ in $V_R^+$. Define $\hat H_\lambda:\hat I_\lambda^+ \to \hat I_\lambda^+$ by 
\[
\hat H_\lambda\left([z,C]\right)=[H_\lambda(z),LH_\lambda(C)].
\]
Since $\pi_\lambda \circ \hat H_\lambda =H_\lambda \circ \pi_\lambda$, $\hat H_\lambda$ is a holomorphic map which is lift of $H$. Let $[z,C],[z',C']\in\hat{I}_\lambda^+$ with $\hat{H}_\lambda\left([z,C]\right)=\hat{H}_\lambda\left([z',C']\right)$. Then $z=z'$ and 
\[
f_{\sigma(\lambda)}\left(H_\lambda(C)\overline{H_\lambda(C')}\right)\in \mathbb{Z}\implies df_\lambda\left(C\bar C'\right)\in \mathbb{Z}.
\]
Thus, $\hat{H}_\lambda$ is a $d$-to-$one$ map.

\begin{prop} $\hat H_\lambda$ satisfies the functional equation $\hat\phi_{\sigma(\lambda)} \circ \hat H_\lambda(\hat{z})=\phi_\lambda(\hat{z})^d$.
\end{prop}
\begin{proof} Let $\hat z =[z,C]\in \hat{I_\lambda^+}$, then
\begin{align*}
    \hat\phi_{\sigma(\lambda)} \circ \hat H_\lambda(\hat{z}) & =\hat\phi_{\sigma(\lambda)} \circ \hat H_\lambda([z,C])\\
    &=\hat\phi_{\sigma(\lambda)}([H_\lambda(z),LH_\lambda(C)])\\
    & =\phi_{\sigma(\lambda)}(\alpha)e^{\int_{LH_\lambda(C)}\omega_{\sigma(\lambda)}}\\
    &=\phi_{\sigma(\lambda)}(H_\lambda(\alpha))e^{\int_C H_\lambda^* \omega_{\sigma(\lambda)}}\\
    &=(\phi_\lambda(\alpha))^d e^{d\int_C \omega_\lambda}\\
    &=\hat{\phi_\lambda}([z,C])^d.
\end{align*}
\end{proof}

\medskip
\noindent
Let $[z,C]\in \hat{V}_R^+$. Since $H(V_R^+)\subset V_R^+$, $\hat{H}\left([z,C]\right)=[H(z),LH(C)]\in \hat{V}_R^+$. Thus we have increasing sequence of open sets $\hat V_R^+ \subset \hat H_\lambda^{-1}(\hat V_R^+) \subset \cdots (\hat H_\lambda^n)^{-1}(\hat V_R^+)\subset\cdots .$
Further, for each $\lambda\in M$,  $I_\lambda^+=\cup_{n\geq 0}(H_\lambda^n)^{-1}(V_R^+)$, thus it follows that
\[
\hat I_\lambda^+=\bigcup_{n=0}^{\infty}\left(\hat H_\lambda^n\right)^{-1}\left(\hat V_R^+\right).
\]

\medskip
\noindent
{\it Step 4.}
Let $U_\lambda^+ =\left\{z=(x,y) \in V_R^+ : \vert \phi_\lambda(z) \vert >M \max\{R, \vert x \vert \} \right\}$. Then using \eqref{eqphi}, we get $H_\lambda(U_\lambda^+) \subset U_{\sigma(\lambda)}^+$. Thus, $U_\lambda^+ \subset H_\lambda^{-1}(U_{\sigma(\lambda)}^+) \subset (H_\lambda^2)^{-1}(U_{\sigma^2(\lambda)}^+) \subset \cdots$.
Note that for some $m\in \mathbb{N}$, $H_\lambda^m(V_R^+)\subset U_{\sigma^m(\lambda)}^+$. Thus, 
\[
I_\lambda^+ = \bigcup_{n \geq 0} (H_\lambda^n)^{-1}(U_{\sigma^n(\lambda)}^+).
\]
Let $\hat U_\lambda^+ = \left\{\hat z \in V_R^+ : \pi_\lambda(\hat z) \in U_\lambda^+ \right\}$. Then
\begin{equation}
\hat U_\lambda^+ \subset \hat H_\lambda^{-1}\left(\hat U_{\sigma(\lambda)}^+\right) \subset \left(\hat H_\lambda^2\right)^{-1}\left(\hat U_{\sigma^2(\lambda)}^+\right) \subset \cdots \text{ and }
    \hat I_\lambda^+ = \bigcup_{n \geq 0} \left(\hat H_\lambda^n\right)^{-1}\left(\hat U_{\sigma^n(\lambda)}^+\right).
\end{equation}

\medskip
\no
Since $\phi_\lambda(x,y)\approx y$ as $\|(x,y)\|\to \infty$ in $V_R^+$, there exist sufficiently large $R>0$ such that the map $(x,y) \mapsto (x,\phi_\lambda(x,y))=(x,t_\lambda)$ is a biholomorphism from $U_\lambda^+$ onto $\{(x,t) \in \mathbb{C}^2 : \vert t \vert >M \max\{R, \vert x \vert \} \}$. Let $(x,t) \mapsto (x,\theta_\lambda(x,t))$ be its inverse. Then
\begin{equation}
    \frac{\partial(x,t_\lambda)}{\partial(x,y)}=\frac{1}{\partial\theta_\lambda/\partial t_\lambda}.
\end{equation}
\noindent
Let $(x,t_\lambda) \mapsto \left(t_\lambda \int_0^x \frac{\partial\theta_\lambda}{\partial t_\lambda}(\zeta,t_\lambda)d\zeta,t_\lambda\right)=(\tilde{s_\lambda},t_\lambda)$, then
\begin{equation}
    \frac{\partial(\tilde{s_\lambda},t_\lambda)}{\partial(x,t_\lambda)}=t_\lambda\frac{\partial\theta_\lambda}{\partial t_\lambda}.
\end{equation}
\noindent
Thus, the Jacobian determinant of the composition map $(x,y)\mapsto(x,t_\lambda)\mapsto (\tilde{s_\lambda},t_\lambda)$ is $t_\lambda$.
Consider $H_\lambda$ in the coordinates $(\tilde{s},t)$, we get
\begin{equation}\label{newH}
    (\tilde{s}_\lambda,t_\lambda)=(\tilde{s},t)\mapsto(x,y)\mapsto H_\lambda(x,y)=(x_1,y_1)\mapsto(\tilde{s}_{\sigma(\lambda),1},t_{\sigma(\lambda),1})=(\tilde{s_1},t_1),
\end{equation}
\noindent
here $t_\lambda=\phi_\lambda(x,y)$ and $t_{\sigma(\lambda),1}=\phi_{\sigma(\lambda)}(H_\lambda(x,y))=\phi_\lambda(x,y)^d = t_\lambda^d$. Hence $t_1=t^d$.
\noindent
Examine the image of $(0,t)$ under the map in \eqref{newH}, we get
\begin{equation}\label{0,t}
    (0,t)\mapsto(0,\theta_\lambda(0,t))\mapsto H_\lambda(0,\theta_\lambda(0,t))\mapsto \left(t^d \int_0^{\theta_\lambda(0,t)}\frac{\partial\theta_{\sigma(\lambda)}}{\partial t}(\zeta,t^d)d\zeta,t^d\right).
\end{equation}
\noindent
Comparing the Jacobian determinant of the map in \eqref{newH} and product of Jacobian determinants of individual maps that are composed to get the map \eqref{newH}, gives
\begin{equation}\label{ode}
    \frac{\partial \tilde{s}_1}{\partial \tilde{s}}=\frac{\delta_\lambda}{d} \Rightarrow \tilde{s}_1(\tilde{s},t)=\frac{\delta_\lambda}{d}\tilde{s} +C_\lambda(t),
\end{equation}
\noindent
where $C_\lambda(t)$ is a holomorphic function on $\vert t \vert>MR$.
Now, \eqref{0,t} and \eqref{ode} give
\[
C_\lambda(t)=t^d \int_0^{\theta_\lambda(0,t)}\frac{\partial\theta_{\sigma(\lambda)}}{\partial t_1}(\zeta,t^d)d\zeta.
\]
Since $\theta_\lambda(0,t)\approx t$ as $t\to \infty$, the Laurent expansion of $C_\lambda$ is  
\[
C_\lambda(t)=t^{d+1}+l_{\lambda,d}t^d+\cdots+l_{\lambda,0}+\frac{l_{\lambda,-1}}{t}+\cdots.
\]
Let $Q_\lambda$ and $C_\lambda^-$ denote the polynomial and singular parts, respectively, of $C_\lambda$, that is, 
\begin{center}
   $Q_\lambda(t)=t^{d+1}+l_{\lambda,d}t^d+\cdots+l_{\lambda,0}$ and  $C_\lambda^-(t)=\frac{l_{\lambda,-1}}{t}+\frac{l_{\lambda,-2}}{t^2}+\cdots$.
\end{center}
\noindent
Thus, $C_\lambda(t)=Q_\lambda(t)+C_\lambda^-(t)$. Define 
\[
R_\lambda(t)=\frac{d}{\delta_\lambda}C_\lambda^-(t)+\left(\frac{d^2}{\delta_\lambda \delta_{\sigma(\lambda)}}\right)C_{\sigma(\lambda)}^-(t^d)+\left(\frac{d^3}{\delta_\lambda \delta_{\sigma(\lambda)} \delta_{\sigma^2(\lambda)}}\right)C_{\sigma^2(\lambda)}^-(t^{d^2})+\cdots.
\]
Let $s=\tilde{s}+R_\lambda(t)$, then the map \eqref{newH} in the new coordinate (s,t) will become
\begin{align*}
(s,t) & \mapsto \left(s-R_{\lambda}(t),t\right) \mapsto \left(\frac{\delta_\lambda}{d}\left(s-R_\lambda(t)\right)+C_\lambda(t),t^d\right)\\
& \mapsto \left(\frac{\delta_\lambda}{d}s+C_\lambda(t)-\frac{\delta_\lambda}{d}R_\lambda(t)+R_{\sigma(\lambda)}(t^d),t^d\right)=\left(\frac{\delta_\lambda}{d}s+Q_\lambda(t),t^d\right),
\end{align*}
here $Q_\lambda(t)=C_\lambda(t)-\frac{\delta_\lambda}{d}R_\lambda(t)+R_{\sigma(\lambda)}(t^d)$ is the polynomial part of $C_\lambda(t).$ 
Let $\psi_\lambda=s$ on $U_\lambda^+$, then 
\begin{equation}\label{psi}
    \psi_{\sigma(\lambda)}(H_\lambda(z))=\frac{\delta_\lambda}{d}\psi_\lambda(z)+Q_\lambda(\phi_\lambda(z)).
\end{equation} 
Inductively, for $n\geq 1$, 
\begin{equation}\label{psin}
     \psi_{\sigma^n(\lambda)}(H_\lambda^n(z))=\frac{\delta_{\sigma^{n-1}(\lambda)}...\delta_\lambda}{d^n}\psi_\lambda(z)+\frac{\delta_{\sigma^{n-1}(\lambda)}...\delta_{\sigma(\lambda)}}{d^{n-1}}Q_\lambda(\phi_\lambda(z))+...+Q_{\sigma^{n-1}(\lambda)}(\phi_\lambda(z)^{d^{n-1}}).
\end{equation}
\noindent
Note that if we choose $R>0$ large enough then the map $(x,y)\mapsto(\tilde{s}_\lambda,t_\lambda)$ is a biholomorphism from $U_\lambda^+$ onto its image. Thus, the map $\Phi_\lambda = (\psi_\lambda,\phi_\lambda):U_\lambda^+ \to \mathbb{C}^2$ is an injective holomorphic map.

\medskip
\noindent
{\it Step 5.}
In this step, we prove that $\hat{I}_\lambda^+$ is biholomorphic to $\mathbb{C}\times\mathbb{C}\backslash\mathbb{\overline{D}}$. First we define a holomorphic function $\hat \psi_\lambda$ on $\hat I_\lambda^+$.
For $[z,C]\in \hat U_\lambda^+$, define 
\[
\hat \psi_\lambda([z,C])=\psi_\lambda\circ \pi_\lambda([z,C])= \psi_\lambda(z),
\]
and for $[z,C] \in \left(\hat H_\lambda^n\right)^{-1}\left(\hat{U}_{\sigma^n(\lambda)}^+\right)$, define 
\begin{align*}
\hat \psi_\lambda([z,C])=\frac{d^n}{\delta_{\sigma^{n-1}(\lambda)}\cdots\delta_\lambda}\hat{\psi}_{\sigma^n(\lambda)}\left(\hat H_\lambda^n([z,C])\right) - & \frac{d^n}{\delta_{\sigma^{n-1}(\lambda)}\cdots\delta_\lambda}Q_{\sigma^{n-1}(\lambda)}\left(\hat{\phi_\lambda}([z,C])^{d^{n-1}}\right)-\\
& \cdots-\frac{d}{\delta_\lambda}Q_\lambda\left(\hat{\phi_\lambda}([z,C])\right).
\end{align*}
It follows from \eqref{psin} that $\hat{\psi}_\lambda$ is well-defined and  satisfies the functional equation 
\[
\hat{\psi}_{\sigma(\lambda)}\left(\hat{H_\lambda}([z,C])\right)=\frac{\delta_\lambda}{d}\hat{\psi_\lambda}([z,C])+Q_\lambda\left(\hat{\phi_\lambda}([z,C])\right).
\]
Let $\hat{\Phi}_\lambda=(\hat \psi_\lambda,\hat{\phi_\lambda}):\hat{I}_\lambda^+\to \mathbb{C} \times\mathbb{C}\backslash\mathbb{\overline{D}}$.
And let $G:M\times\mathbb{C}\times\mathbb{C}\backslash\mathbb{\overline{D}} \to M\times\mathbb{C}\times\mathbb{C}\backslash\mathbb{\overline{D}}$ defined by
\[
G(\lambda,s,t)=\left(\sigma(\lambda),G_\lambda(s,t)\right), \text{ where } G_\lambda(s,t)=\left(\frac{\delta_\lambda}{d}s+Q_\lambda(t),t^d\right)
\]
is a holomorphic map from $\mathbb{C}\times \mathbb{C}\backslash\mathbb{\overline{D}}$ onto $\mathbb{C}\times \mathbb{C}\backslash\mathbb{\overline{D}}$. Then
\begin{align*}
G_\lambda^n(s,t) & =G_{\sigma^{n-1}(\lambda)}\circ\cdots\circ G_\lambda(s,t)\\
& =\left(\frac{\delta_{\sigma^{n-1}(\lambda)}\cdots\delta_\lambda}{d^n}s+\frac{\delta_{\sigma^{n-1}(\lambda)}\cdots\delta_{\sigma(\lambda)}}{d^{n-1}}Q_\lambda(t)+\cdots+Q_{\sigma^{n-1}(\lambda)}(t^{d^{n-1}}),t^{d^n}\right).
\end{align*}
Observe that $\hat \Phi_{\sigma(\lambda)}\circ \hat H_\lambda = G_\lambda \circ \hat \Phi_\lambda$. Inductively, for every $n \in \mathbb{N}$, $\hat \Phi_{\sigma^n(\lambda)}\circ \hat H_\lambda^n = G_\lambda^n \circ \hat \Phi_\lambda$, that is, the following diagram commutes.
\[
\begin{tikzcd}
\hat{I}_\lambda^+
  \arrow[d, "\hat{\Phi}_\lambda"']
  \arrow[r, "\hat{H}_\lambda^n"]
&
\hat{I}^+_{\sigma^n(\lambda)}
  \arrow[d, "\hat{\Phi}_{\sigma^n(\lambda)}"]
\\
\mathbb{C}\times \mathbb{C}\backslash\mathbb{\overline{D}}
  \arrow[r, "G_{\lambda}^n"']
&
\mathbb{C}\times\mathbb{C}\backslash\mathbb{\overline{D}}
\end{tikzcd}
\]
Let $W_\lambda=\hat \Phi_\lambda(\hat U_\lambda^+)$. 
Then $G_\lambda(W_\lambda)\subset W_{\sigma(\lambda)}$.
Thus, $W_\lambda \subset G_\lambda^{-1}(W_{\sigma(\lambda)}) \subset (G_\lambda^2)^{-1}(W_{\sigma^2(\lambda)}) \subset \cdots$.
Further, we claim that
\begin{equation}\label{Wlambda}
\mathbb{C}\times \mathbb{C}\backslash\mathbb{\overline{D}} = \bigcup_{n\geq 0}(G_\lambda^n)^{-1}(W_{\sigma^n(\lambda)}).
\end{equation}
Before proving the claim, we prove that $\{(s,t)\in \mathbb{C}^2:\vert t \vert >\tilde{M},\vert s  \vert < \rho \vert t \vert ^2\}\subset\Phi_\lambda(U_\lambda^+)$, where $\tilde{M}$ (large enough), and $\rho$ (small enough) are constants. The proof of this result follows the approach of \cite[Lemma 2.1]{MP}. Let $\epsilon_0>0$. Since $\phi_\lambda(x,y)\approx y$ as $\|(x,y)\|\to \infty$ in $V_R^+$, $\theta_\lambda(x,t)\approx t$ as $t\to \infty$ in $\{(x,t)\in \mathbb{C}^2: |t|>M\max\{R,|x|\}\}$. Thus, for large $R$ and for $(x,t)\in\{(x,t)\in \mathbb{C}^2:\max\{M|x|,RM\}<|t|\}$,
 \[
  (1-\epsilon_0)<\left|\frac{\partial\theta_\lambda}{\partial t}(x,t)\right|<(1+\epsilon_0)
 \]
Fix $|t|>MR$. Consider the map $\kappa_t: \{x\in \mathbb{C}: |x|<|t|/M\}\to \mathbb{C}$ defined by
\[
 x\mapsto \int_0^x \frac{\partial\theta_\lambda}{\partial t}(\xi,t)\mathrm{d}\xi.
\]
It is a holomorphic map and we choose $M$ large such that it extends continuously to the boundary. For $|x|=|t|/M$, 
\[
|\kappa_t(x)|>(1-\epsilon_0)|t|/M.
\]
Let $x_0\in \mathbb{C}$ be such that $|x_0|<(1-\epsilon_0)|t|/M$. Consider the holomorphic function $\kappa_{t,x_0}:\{x:\mathbb{C}:|x|<|t|/M\}\to\mathbb{C}$ defined by
\[
x\mapsto \kappa_t(x)-x_0.
\]
Then 
\[
|\kappa_{t,x_0}(x)-\kappa_t(x)|=|x_0|<|\kappa_t(x)|
\]
for $|x|=|t|/M$. Thus, by Rouche's theorem, $\kappa_{t,x_0}$ and $\kappa_t$ have same number of zeros in $\{x\in \mathbb{C}^2 : |x|<|t|/M\}$. Since $\kappa_t(0)=0$, there exists $x_1\in \{|x|<|t|/M\}$ such that $\kappa_{t,x_0}(x_1)=0$ i.e, $\kappa_t(x_1)=x_0$. Thus, for $(x,y)\in U_\lambda^+$, the map
\[
(x,y)\mapsto \left(t \int_0^x \frac{\partial\theta_\lambda}{\partial t}(\xi,t)\mathrm{d}\xi,t\right) \text{ ; } t=\phi_\lambda(x,y)
\]
contains the set $\{(s,t)\in \mathbb{C}^2 : |s|<(1-\epsilon_0)|t|^2/M, |t|>MR\}$ in the range. Now choose $\tilde{M}>MR$ large enough such that for $|t|>\tilde{M}$, $|R_\lambda(t)|<(1-\epsilon_0)|t|^2/2M$. Let $|s|<\rho<(1-\epsilon_0)|t|^2/2M$, then
\[
|s-R_\lambda(t)|<(1-\epsilon_0)|t|^2/M.
\]
Thus, there is $(x,y)\in U_\lambda^+$ such that $\Phi_\lambda(x,y)=(s,t)$.  Therefore, \eqref{Wlambda} holds.
Now, we see that $\hat \Phi_\lambda=(\hat \psi_\lambda,\hat \phi_\lambda):\hat I_\lambda^+ \to \mathbb{C} \times \mathbb{C}\backslash\mathbb{\overline{D}}$ is a biholomorphism. First note that $\hat \Phi_\lambda$ maps $(\hat H_\lambda^n)^{-1}(\hat U_{\sigma^n(\lambda)}^+)$ into $(G_\lambda^n)^{-1}(W_{\sigma^n(\lambda)})$. If $\hat z \in (\hat H_\lambda^n)^{-1}(\hat U_{\sigma^n(\lambda)}^+)$, then $G_\lambda^n(\hat \Phi_\lambda(\hat z)) \in W_{\sigma^n(\lambda)}$ as $\hat \Phi_{\sigma^n(\lambda)}\circ \hat H_\lambda^n = G_\lambda^n \circ \hat \Phi_\lambda$, hence $\hat \Phi_\lambda(\hat z) \in (G_\lambda^n)^{-1}(W_{\sigma^n(\lambda)})$.
We now prove that 
\[
\hat \Phi_\lambda:(\hat H_\lambda^n)^{-1}(\hat U_{\sigma^n(\lambda)}^+) \to (G_\lambda^n)^{-1}(W_{\sigma^n(\lambda)})
\]
is a bijection. For injectivity, let $\hat \Phi_\lambda(\hat z_1)=\hat \Phi_\lambda(\hat z_2)$, for $\hat z_1, \hat z_2 \in (\hat H_\lambda^n)^{-1}(\hat U_{\sigma^n(\lambda)}^+)$. Since $\hat \Phi_{\sigma^n(\lambda)}\circ \hat H_\lambda^n = G_\lambda^n \circ \hat \Phi_\lambda$, so $\hat \Phi_{\sigma^n(\lambda)}(\hat H_\lambda^n(\hat z_1))=\hat \Phi_{\sigma^n(\lambda)}(\hat H_\lambda^n(\hat z_2))$. This gives $\hat H_\lambda^n(\hat z_1)=\hat H_\lambda^n(\hat z_2)$ as $\hat \Phi_{\sigma^n(\lambda)}$ is injective on $\hat U_{\sigma^n(\lambda)}^+$. Hence $z_1=z_2$ and $\hat \Phi_\lambda(\hat z_1)=\hat \Phi_\lambda(\hat z_2)$ gives $\hat z_1 =\hat z_2$. Let $\hat z \in \hat U_{\sigma^n(\lambda)}^+$, then the set $\{(\hat H_\lambda^n)^{-1}(\hat z)\}$ has $d^n$ points, and is mapped by $\hat \Phi_\lambda$ into $\{(G_\lambda^n)^{-1}(\hat \Phi_{\sigma^n(\lambda)}(\hat z))\}$ which again has $d^n$ points in it. As this map is injective, so is surjective.

\medskip
\noindent
Since for every $\lambda \in M$, $\hat I^+_\lambda$ is biholomorphic to $\mathbb{C}\times \mathbb{C}\backslash\mathbb{\overline{D}}$, so $I^+_\lambda$ is biholomorphic to quotient of $\mathbb{C}\times \mathbb{C}\backslash\mathbb{\overline{D}}$ by discrete group of covering transformations which is isomorphic to $\mathbb{Z}\left[1/d\right]/\mathbb{Z}$. Hence for any $\lambda_1,\lambda_2 \in M$, $I^+_{\lambda_1}$ and $I^+_{\lambda_2}$ are biholomorphic.

\medskip
\no
{\it Step 6.} In this step, for the special case $M=\overline{\mathbb{D}}$, we construct a covering space of the global escaping set.
Let \[\tilde{V}_R^+ = \{(\lambda,x,y) : \lambda \in \mathbb{D}, (x,y) \in V_R^+ \} = \mathbb{D} \times V^+_R.\]
Since for each $\lambda\in \mathbb{D}$, $V_R^+\subset I^+_\lambda$, $\tilde{V}_R^+ \subset I^+$. If $(\lambda,x,y) \in \tilde{V}_R^+$ then $H(\lambda,x,y)=(\sigma(\lambda),H_\lambda(x,y)) \in \tilde{V}_R^+$, hence we have the increasing sequence $\tilde{V}_R^+ \subset H^{-1}(\tilde{V}_R^+) \subset (H^2)^{-1}(\tilde{V}_R^+) \subset \cdots$. Since for every $\lambda \in M$, $I_{\lambda}^+ = \bigcup_{n=0}^\infty (H_\lambda^n)^{-1}(V_R^+)$, it follows that 
\[I^+ = \bigcup_{n \geq 0} (H^n)^{-1}(\tilde{V}_R^+).\]
Now, to define an analogue of the B\"ottcher function on $\mathbb{D} \times V_R^+$, choose a branch of logarithm $\beta$ such that on $\mathbb{D} \times V_R^+$ 
\[
e^{\beta(\lambda,x,y)} = \frac{q_\lambda(y)-\delta_\lambda x}{y^d} = e^{\beta_\lambda(x,y)}.
\]
Inductively, for $n\geq 1$, 
\[
proj_2 \circ H_\lambda^n(x,y)=y^{d^n} e^{\beta(H^{n-1}(\lambda,x,y))+...+d^{n-1}\beta(\lambda,x,y)}.
\]
 Taking the $d^n$-$th$ root, we get
 \[
 \phi_n(\lambda,x,y)=y e^{\frac{1}{d}\beta(\lambda,x,y)+...\frac{1}{d^n}\beta(H^{n-1}(\lambda,x,y))} = \phi_{\lambda,n}(x,y).
 \]
Let 
\[
\phi(\lambda,x,y)= \lim_{n \to \infty} \phi_n(\lambda,x,y) = ye^{\gamma(\lambda,x,y)} = y e^{\gamma_\lambda(x,y)},
\]
where as discussed in Step~1, the convergence is uniform on compact subsets of $\mathbb{D}\times V^+_R$. So, $\phi:\tilde{V}_R^+ \to \mathbb{C}\backslash\mathbb{\overline{D}}$ is holomorphic as $\beta$ is. Further, for $(\lambda,x,y) \in \tilde{V}_R^+$,
\begin{equation}\label{eq2.15}
\phi(H(\lambda,x,y))=\phi(\sigma(\lambda),H_\lambda(x,y))=\phi_{\sigma(\lambda)}(H_\lambda(x,y))=(\phi_\lambda(x,y))^d=(\phi(\lambda,x,y))^d.
\end{equation}
As $y \to \infty$ in $\tilde V_R^+$, $\phi \approx y$.
Define closed one form $\omega=d\phi/\phi$ on $\tilde{V}_R^+$ and extend it to $I^+$ using \eqref{eq2.15}, i.e., 
\[\omega=\frac{1}{d^n}(H^n)^*(\omega) \quad \text{on } (H^n)^{-1}(\tilde{V}_R^+).\]
Then $\omega$ is well defined on $I^+$ with $H^*\omega=d \omega$. Note that $\omega(\lambda,x,y)=\omega_\lambda(x,y)$.
The fundamental group of $\tilde{V}_R^+$ is $\pi_1(\tilde{V}_R^+) \cong \mathbb{Z}$ and $f:\pi_1(I^+) \to \mathbb{Z}[1/d]$ defined by $f([C])=\int_C \omega$ is a group isomorphism. 
Fix $\lambda_0 \in \mathbb{D}$ and $\alpha \in V^+_R$. Define a relation $\sim$ on the set $A=\{(\lambda,z,C): (\lambda,z) \in I^+, C$ is a path from $(\lambda_0,\alpha)$ to $(\lambda,z)$ in $I^+\}$ by 
\[
(\lambda,z,C) \sim (\lambda',z',C') \iff \lambda=\lambda',z=z',f([C \bar{C'}]) \in \mathbb{Z}.
\]
Let $\hat I^+ =A/\sim$ and let $\pi:\hat I^+ \to I^+$ defined by $[(\lambda,z,C)] \mapsto (\lambda,z)$. Give $\hat I^+$ topology which makes $\pi$ a continuous map. Further, since $\pi$ is a covering map we can pullback analytic structure of $I^+$ to give an analytic structure on $\hat I^+$ such that $\pi$ is holomorphic.

\medskip
\noindent
Let $C$ be a closed curve based at $(\lambda,z)$ in $I^+$. Then for some $n\geq 1$, $H^n(C) \subset \tilde{V}_R^+$, which implies $(proj_2 \times proj_3) \circ H^n(C) \subset V_R^+$. Let $\gamma_0=(H_\lambda^n)^{-1}((proj_2 \times proj_3) \circ H^n (C))$. Since $[H^n(C)]=[H^n(\lambda,\gamma_0)]$, $[C]=[(\lambda,\gamma_0)]$ in $I^+$.  
Thus, 
\[
\hat{I}^+= \{(\lambda,\hat{z}) : \lambda \in \mathbb{D}, \hat{z} \in \hat{I}^+_\lambda\}.
\]

\medskip
\noindent
{\it Step 7.} We now prove that $\hat{I}^+$ is biholomorphic to $\mathbb{D}\times\mathbb{C}\times\mathbb{C}\backslash\mathbb{\overline{D}}$.
Let 
\begin{align*}
    U^+ & =\{(\lambda,x,y) :\lambda\in \mathbb{D},  |\phi_\lambda(x,y)|>M \max\{R,|x|\}\}\\
    & =\{(\lambda,x,y):\lambda \in \mathbb{D}, (x,y)\in U^+_\lambda\},
\end{align*}
then
\begin{center}
$U^+ \subset H^{-1}(U^+) \subset (H^2)^{-1}(U^+) \subset \cdots.$
\end{center}
Consider the biholomorphism \[(\lambda,x,y) \mapsto (\lambda,x,\phi(\lambda,x,y))=(\lambda,x,t)\] from $U^+$ onto $\{(\lambda,x,t) \in \mathbb{D} \times \mathbb{C} \times \mathbb{C}\backslash\mathbb{\overline{D}} : |t|>M \max\{R,|x|\} \}$, with inverse \[(\lambda,x,t) \mapsto (\lambda,x,\theta(\lambda,x,t)).\] 
Consider the map $(\lambda,x,t) \mapsto(\lambda,\tilde{s},t)$, where
\[\tilde{s}=t \int_0^x \frac{\partial \theta}{\partial t}(\lambda,\xi,t)d\xi.\]
Then $H$ in $(\lambda,\tilde{s},t)$ coordinates is
\[(\lambda,\tilde{s},t)\mapsto (\lambda,x,y) \mapsto H(\lambda,x,y) \mapsto (\sigma(\lambda),\tilde{s}_1,t_1),\]
here $t_1=t^d$. Comparing the Jacobian determinants, we get 
\[d\sigma'(\lambda)\frac{\partial \tilde{s}_1}{\partial \tilde{s}} t^{d-1}=\delta_\lambda \sigma'(\lambda)t^{d-1},\]
which gives $\tilde{s}_1 = (\delta_\lambda/d) \tilde{s} +C(\lambda,t)$. The image of $(\lambda,0,t)$ under the map $(\lambda,\tilde{s},t) \mapsto (\sigma(\lambda),\tilde{s}_1,t_1)$ is
\[\left(\sigma(\lambda), t^d \int_0^{\theta(\lambda,0,t)} \frac{\partial \theta}{\partial t}(\sigma(\lambda),\xi,t^d)d\xi,t^d\right).\] Hence \[C(\lambda,t)=  t^d \int_0^{\theta(\lambda,0,t)} \frac{\partial \theta}{\partial t}(\sigma(\lambda),\xi,t^d)d\xi=C_\lambda(t)\] is holomorphic for $\lambda \in \mathbb{D}$, $|t|>MR$. Let $R(\lambda,t)=R_\lambda(t)$ and $s=\tilde{s}+R(\lambda,t)$. Then $H$ in $(\lambda,s,t)$ coordinates is \[(\lambda,s,t) \mapsto \left(\sigma(\lambda), \frac{\delta_\lambda}{d} s +Q(\lambda,t),t^d\right)\] where $Q(\lambda,t)=C(\lambda,t)+R(\sigma(\lambda),t)-(\delta_\lambda/d)R(\lambda,t)=Q_\lambda(t)$. Let $\psi=s$, then $\psi(\lambda,x,y)=\psi_\lambda(x,y)$ and \[\psi\left(H(\lambda,x,y)\right)=\frac{\delta_\lambda}{d} \psi(\lambda,x,y) +Q\left(\lambda,\phi(\lambda,x,y)\right).\]
The map \[(\lambda,x,y) \mapsto(\sigma(\lambda),\psi(\lambda,x,y),\phi(\lambda,x,y))\] is an injective holomorphic map on $U^+$.
Let $\hat{U}^+=\{(\lambda,\hat{z}) : \lambda \in \mathbb{D}, \hat{z} \in \hat{U}^+_\lambda$\}. Consider the map $\hat{H}: \hat{I}^+ \to \hat{I}^+$ defined by $\hat{H}(\lambda,\hat{z})=\left(\sigma(\lambda),\hat{H}_\lambda(\hat{z})\right)$. Let $\hat{\tilde{V}}^+ = \{[(\lambda,z,C)] \in \hat{I}^+ : C$ is a path in $\tilde{V}^+ \}$. Then it follows that 
\begin{center}
$\hat{I}^+=\bigcup_{n \geq 0} (\hat{H}^n)^{-1}(\hat{\tilde{V}}^+)$ and $\hat{I}^+=\bigcup_{n \geq 0} (\hat{H}^n)^{-1}(\hat{U}^+)$.
\end{center}
\noindent
Define $\hat{\phi}:\hat{I}^+ \to \mathbb{C}\backslash\mathbb{\overline{D}}$ by
\[
(\lambda,\hat{z}) \mapsto \hat{\phi}_\lambda (\hat{z})
\]
and $\hat{\psi}:\hat{I}^+ \to \mathbb{C}$ by
\[
\hat{\psi}(\lambda,\hat{z})=\hat{\psi}_\lambda(\hat{z}).
\]
Since $\phi$ and $\psi$ are holomorphic in $\lambda$, so are $\hat{\phi}$ and $\hat{\psi}$. Further, $\hat{\phi} \circ \hat{H}(\lambda,\hat{z}) = \left(\hat{\phi}(\lambda,\hat{z})\right)^d$ and $\hat{\psi} \circ \hat{H}(\lambda,\hat{z}) = (\delta_\lambda/d) \hat{\psi}(\hat{z}) + Q\left(\lambda,(\hat{\phi}(\hat{z})^d\right).$ Define $\hat \Phi : \hat{I}^+ \to \mathbb{D} \times \mathbb{C} \times \mathbb{C} \backslash \mathbb{\overline{D}}$ by 
\[
\hat{\Phi}(\lambda,\hat{z}) = (\lambda,\hat{\psi}(\lambda,\hat{z}), \hat{\phi}(\lambda,\hat{z})),
\]
and $\tilde{G}:\mathbb{D} \times \mathbb{C} \times \mathbb{C} \backslash \mathbb{\overline{D}} \to \mathbb{D} \times \mathbb{C} \times \mathbb{C} \backslash \mathbb{\overline{D}}$ by 
\[
\tilde{G}(\lambda,s,t)=\left(\sigma(\lambda),G_{\lambda}(s,t)\right).
\]
Then for $n\geq 1$, 
\[
\hat{\Phi} \circ \hat{H}^n (\lambda,\hat{z})= (\sigma^{n}(\lambda),G^n_\lambda \circ \hat{\Phi}_\lambda(\hat{z})) = \tilde{G}^n \circ \hat{\Phi}(\lambda,\hat{z}).\]
Let $W= \{(\lambda,s,t) : \lambda\in \mathbb{D}, (s,t) \in W_\lambda\}=\hat{\Phi}(\hat U^+)$. Then $\tilde{G}(W) \subset W$, which gives $W \subset \tilde{G}^{-1}(W) \subset (\tilde{G}^2)^{-1}(W) \subset \cdots$. Further, since we have fiberwise equality, $ \mathbb{D} \times \mathbb{C} \times \mathbb{C} \backslash \overline{\mathbb{D}} = \bigcup_{n \geq 0}(\tilde{G}^n)^{-1}(W)$.
Now we see that $\hat{\Phi}$ is a biholomorphism by showing that for every $n\geq 1$, the map $\hat{\Phi}|_{(\hat{H}^n)^{-1}(\hat{U}^+)} : (\hat{H}^n)^{-1}(\hat{U}^+) \to (\tilde{G}^n)^{-1}(W)$ is bijective. Injectivity follows since, for every $\lambda \in \mathbb{D}$, $\hat{\Phi}_\lambda$ is injective . For surjectivity, note that for $(\lambda,\hat{z}) \in \hat U^+$, $\{(\hat H^n)^{-1}((\lambda,\hat z))\}$ has $d^n$ elements which by $\hat{\Phi}$, are mapped injectively to $\{(\tilde{G}^n)^{-1}(\hat \Phi(\lambda,\hat{z}))\}$ which also contains $d^n$ elements. 

\section{Proof of Theorem \ref{thm2}}
{\it Step 1.} In this step, through a lemma, we determine the form of the deck transformations of the covering map $(\hat{\Phi})^{-1}\circ \pi:\mathbb{D}\times\mathbb{C}\times\mathbb{C}\backslash\mathbb{\overline{D}}\to I^+$.
\begin{lem} For every element $[\frac{k}{d^n}] \in \mathbb{Z}[1/d]/\mathbb{Z}$, there exists a unique deck transformation \[\gamma_{\frac{k}{d^{n}}}: \mathbb{D} \times \mathbb{C} \times \mathbb{C}\backslash\mathbb{\overline{D}} \to \mathbb{D} \times \mathbb{C} \times \mathbb{C}\backslash\mathbb{\overline{D}}\] defined as
\[(\lambda,s,t)\mapsto \left(\lambda,s+\frac{d}{\delta_\lambda}\sum_{l=0}^{\infty} \left(\frac{d^l}{\delta_{\sigma(\lambda)}\cdots\delta_{\sigma^{l}(\lambda)}}\right) \left[Q_{\sigma^l(\lambda)}(t^{d^l})-Q_{\sigma^l(\lambda)}((e^{2 \pi i \frac{k}{d^{n}}}t)^{d^l})\right],e^{2 \pi i \frac{k}{d^{n}}}t \right),\]
with the convention that for $l=0$, $\delta_{\sigma(\lambda)}\cdots\delta_{\sigma^{l}(\lambda)}=1$.
\end{lem}
\begin{proof}
By Theorem(\ref{tmh1}), the group of deck transformations of the covering map $\hat{\Phi}^{-1}\circ \pi : \mathbb{D} \times\mathbb{C} \times \mathbb{C}\backslash \mathbb{\overline{D}} \to I^+$ is isomorphic to $\mathbb{Z}[1/d]/\mathbb{Z}$ and the lift of H is given by $\tilde{H}(\lambda,s,t) = \left(\sigma(\lambda),(\delta_\lambda/d)s+Q(\lambda,t),t^d\right)$. First, we construct the deck transformation corresponding to the class $[1/d^{n}]$ by induction. For $n=0$, $\gamma_1$ is the identity map. Assume that $\gamma_{\frac{1}{d^{n}}}$ has the form given in the lemma. Since $H$ induces identity on the fundamental group of the global escaping set, $\tilde{H} \circ \gamma_{\frac{1}{d^{n+1}}}(\lambda,s,t)=\gamma_{\frac{1}{d^{n}}} \circ \tilde{H}(\lambda,s,t)$. Let $(\lambda_1,s_1,t_1)=\gamma_{1/d^{n+1}}(\lambda,s,t)$. Then $\tilde{H}(\lambda_1,s_1,t_1)=\gamma_{\frac{1}{d^{n}}}\circ \tilde{H}(\lambda,s,t)$ gives
\begin{align*}
    \sigma(\lambda_1)&=\sigma(\lambda)\implies \lambda_1=\lambda,\\
    \left(\frac{\delta_{\lambda_1}}{d}\right)s_1+Q_{\lambda_1}(t_1)&=\left(\frac{\delta_\lambda}{d}\right)s+Q_\lambda(t)\\
    & +\frac{d}{\delta_{\sigma(\lambda)}}\sum_{l=0}^{\infty}\left(\frac{d^l}{\delta_{\sigma^2(\lambda)}\cdots\delta_{\sigma^{l+1}(\lambda)}}\right)\left[Q_{\sigma^{l+1}(\lambda)}(t^{d^{l+1}})-Q_{\sigma^{l+1}(\lambda)}((e^{2 \pi i \frac{1}{d^{n+1}}}t)^{d^{l+1}})\right],\\
    t_1^d & =e^{2\pi i \frac{1}{d^{n}}}t^d.
\end{align*}
The above equations give the form of $\gamma_{\frac{1}{d^{n+1}}}$. Further, $\gamma_{\frac{k+1}{d^{n}}}=\gamma_{\frac{1}{d^{n}}}\circ \gamma_{\frac{k}{d^{n}}}$ proves the lemma.
\end{proof}

\medskip
\no
{\it Step 2.} 
For $|y|>R$, consider the map $S:(\lambda,x,y) \mapsto \left(\lambda,x/y,1/y\right)=(\lambda,\xi,\eta)$ with inverse $(\lambda,\xi,\eta) \mapsto \left(\lambda,\xi/\eta,1/\eta\right)$. Then $H$ in $(\lambda,\xi,\eta)$ coordinates is
\begin{align*}
(\lambda,\xi,\eta) & \mapsto \left(\lambda,\frac{\xi}{\eta}, \frac{1}{\eta}\right) \mapsto \left(\sigma(\lambda),\frac{1}{\eta},q_\lambda\left(\frac{1}{\eta}\right)-\delta_\lambda \left(\frac{\xi}{\eta}\right)\right)\\
& \mapsto \left(\sigma(\lambda),\frac{\eta^{d-1}}{\eta^d q_\lambda(1/\eta)-\delta_\lambda \xi \eta^{d-1}}, \frac{\eta^d}{\eta^d q_\lambda(1/\eta)-\delta_\lambda \xi \eta^{d-1}}\right).
\end{align*}
\noindent
Recall from section~2 the coordinate change $(\lambda,x,y) \mapsto (\lambda,x,\phi(\lambda,x,y))=(\lambda,x,t)$. Then
\begin{align*}
\frac{1}{t}= \frac{1}{\phi(\lambda,x,y)}=\frac{1}{\phi(\lambda,\xi/\eta,1/\eta)}=\eta e^{-\gamma(\lambda,\xi/\eta,1/\eta)} = \eta X(\lambda,\xi,\eta),
\end{align*}
where $X$ is holomorphic on $\{(\lambda,\xi,\eta) : \lambda \in \mathbb{D}, 0< |\eta|<1/R, |\xi|<1\}$. Since $\phi(\lambda,x,y) \approx y$ as $y \to \infty$, we have $X(\lambda,\xi,0)= 1$ for every $\lambda$, $\xi$. Thus, 
\[
X(\lambda,\xi,\eta)=1+\eta \alpha(\lambda,\xi,\eta)
\]
where $\alpha(\lambda,\xi,\eta)$ is a power series in the variables $\lambda$, $\xi$, and $\eta$.
Consider 
\begin{align*}
(\lambda,\xi,\eta) \mapsto \left(\lambda,\frac{\xi/\eta}{\phi(\lambda,\xi/\eta,1/\eta)},\frac{1}{\phi(\lambda,\xi/\eta,1/\eta)}\right)=(\lambda,x/t,1/t)=(\lambda,u,v),
\end{align*}
where $v=\eta+\eta^2 \alpha(\lambda,\xi,\eta)$ and $\eta=v+T_1(\lambda,u,v),$ such that the power series of $T_1$  has total degree $\geq 2$ in $u$ and $v$.
Thus,
\begin{align*}
\frac{1}{t}=v & =\eta+\eta^2\alpha(\lambda,\xi,\eta)\\
&=\eta\left[1+\left(v+T_1(\lambda,u,v)\right)\alpha(\lambda,u,v)\right]\\
& =\eta\left[1+D_{\lambda,0} v+\alpha_1(\lambda,u,v)\right],
\end{align*}
where $\alpha_1$ is power series in $\lambda$, $u$, and $v$ such that the total degree in $u$ and $v$ of the terms in the series is $\geq 2$.
Hence, we have 
\[
y=t\left(1+\frac{c_\lambda}{t}+\alpha_1(\lambda,x/t,1/t)\right),
\]
and for $D_{\lambda,0}, A_{\lambda,i}^j\in \mathbb{C}$,
\[
y(\lambda,x,t)=t+D_{\lambda,0}+ \frac{A^0_{\lambda,-1}}{t}+\cdots+x\left(\frac{A^1_{\lambda,-1}}{t}+\cdots\right)+\cdots.
\]
Therefore, for $D_{\lambda,i}\in \mathbb{C}$,
\[
y(\lambda,0,t)=y(\lambda,t)=t+D_{\lambda,0}+\frac{D_{\lambda,1}}{t}+\frac{D_{\lambda,2}}{t^2}+\cdots.\]
Since $t=\phi(\lambda,x,y)\approx y$ as $y \to \infty$, for $L_{\lambda,i}\in \mathbb{C}$,
\[
t(\lambda,0,y)=t(\lambda,y)=y+\frac{L_{\lambda,1}}{y}+\frac{L_{\lambda,2}}{y^2}+\cdots,
\]
which gives
\[
y=\left(y+\frac{L_{\lambda,1}}{y}+\cdots\right)+D_{\lambda,0}+D_{\lambda,1}\left(y+\frac{L_{\lambda,1}}{y}+\cdots\right)^{-1}+\cdots.\]
Comparing the coefficients, we obtain $D_{\lambda,0}=0$. Hence we have 
\[
y=t+\frac{D_{\lambda,1}}{t}+\frac{D_{\lambda,2}}{t^2}+\cdots.
\]
Further, note that $Q(\lambda,t)$ is polynomial part of $t^d y(\lambda,0,t)$. Thus, $Q(\lambda,t)=t^{d+1}+D_{\lambda,1}t^{d-1}+D_{\lambda,2}t^{d-2}+\cdots+D_{\lambda,d}$.

\medskip
\no
{\it Step 3.} Before proceeding further, we fix the notation. Let $I_H^+$ and $I_F^+$ be the global escaping sets of $H$ and $F$ respectively. The maps $p_H:\mathbb{D} \times \mathbb{C} \times \mathbb{C}\backslash\mathbb{\overline{D}}\to I^+_H$ and $p_F:\mathbb{D} \times \mathbb{C} \times \mathbb{C}\backslash\mathbb{\overline{D}}\to I^+_F$ are coverings constructed in the above section.
Let $\Gamma:I_{H}^+ \to I_{F}^+$ be a biholomorphism, which lifts to a biholomorphism $\tilde{\Gamma}:\mathbb{D} \times \mathbb{C} \times \mathbb{C}\backslash\mathbb{\overline{D}} \to \mathbb{D} \times \mathbb{C} \times \mathbb{C}\backslash\mathbb{\overline{D}}$. First, we determine the form of automorphisms of $\mathbb{D} \times \mathbb{C} \times \mathbb{C}\backslash\mathbb{\overline{D}}$. 
Let $\tilde{\Gamma}(\lambda,s,t) = (A_1(\lambda,s,t),A_2(\lambda,s,t),A_3(\lambda,s,t))$. Fix $\lambda$ and $t$, then the maps 
\begin{center} 
$s \mapsto A_1(\lambda,s,t)$ and $s \mapsto A_3(\lambda,s,t)$
\end{center}
\noindent
from $\mathbb{C} \to \mathbb{D}$ and $\mathbb{C} \to \mathbb{C}\backslash\mathbb{\overline D}$ respectively, are holomorphic, hence constant. Thus, $A_1$ and $A_3$ are independent of $s$.
Let $\tilde \Gamma^{-1} : \mathbb{D} \times \mathbb{C} \times \mathbb{C}\backslash \mathbb{\overline D} \to \mathbb{D} \times \mathbb{C} \times \mathbb{C}\backslash\mathbb{\overline D}$ defined by 
\[
(\lambda,s,t) \to (A^*_1(\lambda,t),A^*_2(\lambda,s,t),A^*_3(\lambda,t))
\]
be the inverse of $\tilde{\Gamma}$. Then
\[\left(A^*_1(A_1(\lambda,t),A_3(\lambda,t)),A^*_2(A_1(\lambda,t),A_2(\lambda,s,t),A_3(\lambda,t)),A^*_3(A_1(\lambda,t),A_3(\lambda,t))\right)=(\lambda,s,t).
\]
Comparing second coordinate we get that for fixed $\lambda$ and $t$, the maps 
\[
s \mapsto A_2(\lambda,s,t)
\text{ and }
s \to A^*_2(A_1(\lambda,t),s,A_3(\lambda,t))
\]
are inverses of each other.
Hence, the map 
\[
s \mapsto A_2(\lambda,s,t)
\] 
from $\mathbb{C} \to \mathbb{C}$ is a biholomorphism. Therefore, it is of the form
\[
s \mapsto \beta'(\lambda,t)s+\gamma'(\lambda,t)
\]
where $\beta'$ and $\gamma'$ are holomorphic functions. Thus,
\[
\tilde{\Gamma}(\lambda,s,t)=(A_1(\lambda,t),\beta'(\lambda,t)s+\gamma'(\lambda,t),A_3(\lambda,t)).
\]
\noindent
Now, let $p \in I^+_H$ and $\tilde{p}=(\lambda,s,t) \in \mathbb{D} \times \mathbb{C} \times \mathbb{C}\backslash\mathbb{\overline D}$ be in fiber of $p$ above $p_{H}$. Since the cover is normal, all other elements in the fiber of $p$ are images of deck transformations at $\tilde{p}$, that is, 
\[
p_{H}^{-1}(p)=\biggl\{\gamma_{H,{\frac{k}{d^{n}}}}(\tilde{p}) : k \geq 1,n \geq 0 \biggr\},
\] 
where
\begin{align*}
\gamma_{H,\frac{k}{d^{n}}}(\lambda,s,t) = & \left( \lambda, s+\frac{d}{\delta_{H,\lambda}}\sum_{l=0}^{\infty} \left(\frac{d}{\delta_H}\right)^l \left[Q_{H,\lambda}(t^{d^l})- Q_{H,\lambda}((e^{2 \pi i \frac{k}{d^{n}}}t)^{d^l})\right],e^{2 \pi i \frac{k}{d^{n}}}t \right).
\end{align*}
\noindent
Similarly for F, we have deck transformations $\gamma_{F,\frac{k}{d^{n}}}$. Since $\Gamma \circ p_H = p_F \circ \tilde{\Gamma}$, if two elements are in fiber of $p$ above $p_H$ then their images under the biholomorphism $\tilde{\Gamma}$ are in the fiber of $\Gamma(p)$ above $p_F$. Note that $\tilde{\Gamma}\circ \gamma_{H,\frac{k}{d^{n}}}\circ\tilde{\Gamma}^{-1}$ is a deck transformation of $p_F$. Therefore, we obtain
\[
\tilde{\Gamma}\circ\gamma_{H,\frac{k}{d^{n}}}=\gamma_{F,\Phi\left(\frac{k}{d^{n}}\right)}\circ\tilde{\Gamma}.
\]
Now, we prove that $\Phi$ is a group homomorphism. Let $k/d^n, k_1/d^{n_1}\in \mathbb{Z}[1/d]/\mathbb{Z}$, then
\begin{align*}
\tilde{\Gamma}\circ\gamma_{H,\frac{k}{d^n}+\frac{k_1}{d^{n_1}}} &=\tilde{\Gamma}\circ\gamma_{H,\frac{k}{d^n}}\circ\gamma_{H,\frac{k_1}{d^{n_1}}}\\
&=\gamma_{F,\Phi\left(\frac{k}{d^n}\right)}\circ\tilde{\Gamma}\circ\gamma_{H,\frac{k_1}{d^{n_1}}}\\
&= \gamma_{F,\Phi\left(\frac{k}{d^n}\right)}\circ\gamma_{F,\Phi\left(\frac{k_1}{d^{n_1}}\right)}\circ\tilde{\Gamma}\\
&=\gamma_{F,\Phi\left(\frac{k}{d^n}\right)+\Phi\left(\frac{k_1}{d^{n_1}}\right)}\circ\tilde{\Gamma}.
\end{align*}
Thus,
\[
\Phi\left(\frac{k}{d^n}+\frac{k_1}{d^{n_1}}\right)=\Phi\left(\frac{k}{d^n}\right)+\Phi\left(\frac{k_1}{d^{n_1}}\right).
\]
Thus $\Phi$ is a group homomorphism. Applying the same argument to $\tilde{\Gamma}^{-1}$ gives that $\Phi$ is a group isomorphism.
Since every deck transformation fixes the first coordinate, comparing the first coordinates in $\tilde{\Gamma}\circ\gamma_{H,\frac{k}{d^{n}}}=\gamma_{F,\Phi\left(\frac{k}{d^{n}}\right)}\circ\tilde{\Gamma}$
we obtain
\[
A_1(\lambda,e^{2\pi ik/d^n}t)=A_1(\lambda,t).
\]
Therefore, $A_1(\lambda,t)=A_1(\lambda)$ is independent of $t$. Note that $A_1:\mathbb{D}\to \mathbb{D}$ is a biholomorphism. Further, the map
\[
(\lambda,t)\mapsto(A_1(\lambda),A_3(\lambda,t))
\]
is a biholomorphism from $\mathbb{D}\times\mathbb{C}\backslash\mathbb{\overline{D}}$ onto itself.
Fix $\lambda\in\mathbb{D}$.
Then
\[
t\mapsto A_3(A_1^{-1}(\lambda),t)
\]
is a biholomorphic self-map of
$\mathbb C\backslash\overline{\mathbb{D}}$. Since every automorphism of
$\mathbb C\backslash\overline{\mathbb{D}}$ is a rotation,
\[
A_3(A_1^{-1}(\lambda),t)=e^{i\alpha_3(A_1^{-1}(\lambda))}t,
\]
where $\alpha_3:\mathbb{D}\to\mathbb{R}$ is a real-valued function. Since $A_3$ is holomorphic in $\lambda$, $\alpha_3$ is constant. Thus, $A_3(\lambda,t)=A_3t$, where $|A_3|=1$.  
Further, comparing the third coordinate of $\gamma_{F,\Phi(\frac{k}{d^n})} \circ \tilde{\Gamma}= \tilde{\Gamma} \circ \gamma_{H,\frac{k}{d^n}}$, we get $e^{2\pi i k/d^n}A_3t= e^{2\pi i \Phi(k/d^n)}A_3 t$. Thus, $\Phi$ is identity map. Thus, $\gamma_{F,\frac{k}{d^n}} \circ \tilde{\Gamma}= \tilde{\Gamma} \circ \gamma_{H,\frac{k}{d^n}}$
An easy computation gives 
\begin{equation}\label{eq3.2}
\begin{split}
proj_2 \circ \tilde{\Gamma}\circ \gamma_{H,\frac{k}{d^n}}(\lambda,s,t)=& \beta'(\lambda,e^{2 \pi i k/d^n}t)s+\gamma'(\lambda,e^{2\pi i k/d^n}t)\\
&+\beta'(\lambda,e^{2 \pi i k/d^n}t)\frac{d}{\delta_{H}}\sum_{l=0}^{\infty}\left(\frac{d}{\delta_H}\right)^l\cdot
\left[Q_{H,\lambda}(t^{d^l})-Q_{H,\lambda}((e^{2 \pi i k/d^n}t)^{d^l})\right],
\end{split}
\end{equation}
\noindent
and
\begin{equation}\label{eq3.5}
\begin{split}
proj_2 \circ \gamma_{F,\frac{k}{d^n}}\tilde{\Gamma}(\lambda,s,t)= \beta'(\lambda,t)s+& \gamma'(\lambda,t)+\frac{d}{\delta_F}\sum_{l=0}^{\infty}\left(\frac{d}{\delta_F}\right)^l\cdot\\
& \left[Q_{F,A_1(\lambda)}((A_3t)^{d^l})-Q_{F,A_1(\lambda)}((e^{2 \pi i k/d^n}A_3t)^{d^l})\right].
\end{split}
\end{equation}
\noindent
Comparing \eqref{eq3.2} and \eqref{eq3.5} we get $\beta'(\lambda,e^{2\pi ki/d^n}t)=\beta'(\lambda,t)$. So $\beta'$ is independent of $t$, that is, $\beta'(\lambda,t)=\beta'(\lambda)$.
Let
\begin{align*}
\Delta_H(\lambda,t)= &\beta'(\lambda)\frac{d}{\delta_{H}}\sum_{l=0}^{\infty}\left(\frac{d}{\delta_H}\right)^l\left[Q_{H,\lambda}(t^{d^l})-Q_{H,\lambda}((e^{2 \pi i /d^n}t)^{d^l})\right],
\end{align*}
 and 
\begin{align*}
 \Delta_F(\lambda,t) = \frac{d}{\delta_F}\sum_{l=0}^{\infty}\left(\frac{d}{\delta_F}\right)^l\left[Q_{F,A_1(\lambda)}((A_3 t)^{d^l})-Q_{F,A_1(\lambda)}((e^{2 \pi i /d^n}A_3 t)^{d^l})\right].
\end{align*}
Equations \eqref{eq3.2} and \eqref{eq3.5} give $\Delta_H(\lambda,t)-\Delta_F(\lambda,t)=\gamma'(\lambda,e^{2 \pi i /d^n}t)-\gamma'(\lambda,t)$.
If we fix $\lambda$ and $t$, then $\gamma'(\lambda,e^{2 \pi i /d^n}t)-\gamma'(\lambda,t)$ is bounded for $n \geq 1$ and hence $\Delta_H(\lambda,t)-\Delta_F(\lambda,t)$ is bounded for $n \geq 1$. For $D_{H,\lambda,i},D_{F,\lambda,i}\in \mathbb{C}$, $Q_{H,\lambda}(t) = t^{d+1}+D_{H,\lambda,1}t^{d-1}+\cdots+D_{H,\lambda,d}$ and $Q_{F,\lambda}(t) = t^{d+1}+D_{F,\lambda,1}t^{d-1}+\cdots+D_{F,\lambda,d}$.
The highest degree term of $\Delta_H(\lambda,t)$ is 
\[
\beta'(\lambda)\left(\frac{d}{\delta_H}\right)^n(1-e^{2 \pi i /d})t^{d^{n-1}(d+1)}
\]
and of $\Delta_F(\lambda,t)$ is 
\[
\left(\frac{d}{\delta_F}\right)^n(1-e^{2 \pi i /d})A_3^{d^{n-1}(d+1)}t^{d^{n-1}(d+1)}.
\]
Fix $\lambda \in \mathbb{D}$. Since $\Delta_H(\lambda,t)-\Delta_F(\lambda,t)$ is bounded for $n \geq 1,$
\[
\lim_{n\to \infty}\left(\frac{\delta_H}{\delta_F}\right)^n A_3^{(d+1)d^{n-1}}=\beta'(\lambda) \text{  and  }\lim_{n\to \infty}\left(\frac{\delta_H}{\delta_F}\right)^{n+1} A_3^{(d+1)d^n}=\beta'(\lambda).
\]
Dividing, we get 
\[
\lim_{n\to \infty}\left(\frac{\delta_H}{\delta_F}\right) A_3^{(d^2-1)d^{n-1}}=1,
\] 
and 
\[
\lim_{n\to \infty}\left(\frac{\delta_H}{\delta_F}\right) A_3(\lambda)^{(d^2-1)d^n}=1.
\] 
Again, dividing, we get 
\[
\lim_{n\to \infty}A_3^{(d^2-1)(d-1)d^{n-1}}=1,
\]
which gives $\delta_H^{d-1}=\delta_F^{d-1}$.
First, we tackle the case where $\delta_{H}=\delta_{F}=\delta$. In this case, the highest degree term of $\Delta_H(\lambda,t)-\Delta_F(\lambda,t)$ is \[\left(\frac{d}{\delta}\right)^n(1-e^{2\pi i/d})\left(\beta'(\lambda)-A_3^{d^{n-1}(d+1)}\right)t^{d^{n-1}(d+1)}.\]
Hence, 
\[
A_3^{d^{n-1}(d+1)} \to \beta'(\lambda)
\]
as $n \to \infty$ otherwise, $\Delta_H(\lambda,t)-\Delta_F(\lambda,t) \to \infty$ as $n \to  \infty$. But  $\Delta_H(\lambda,t)-\Delta_F(\lambda,t)$ is bounded for $n\geq 1$. Note that $\beta'(\lambda)=\beta'$ is independent of $\lambda$. Further,
\begin{center}
$A_3^{d^{n}(d+1)} \to \beta'$ and $A_3^{d^{n}(d+1)} \to \beta'^d$ as $n \to \infty$,
\end{center}
gives $\beta'^{d-1}=1$. As observed in \cite{Pal}, since $\beta'$ is a repelling fixed point of $z\mapsto z^d$, for $n_0$ large enough, 
\begin{center}
$A_3^{d^n(d+1)}=\beta' \implies A_3^{d^n(d+1)(d-1)}=1$ for $n \geq n_0$.
\end{center}
\noindent
For some $k_1$ with $1 \leq k_1 \leq d-1$, $\beta'=e^{2\pi i k_1 / (d-1)}$, and $A_3^{d^{n_0}}=e^{2\pi i k_1 / (d^2-1)}$. Let $A_4=e^{2\pi i k_1 / (d^2-1)} \implies A_4^{d+1}=\beta'$. And for some $1\leq k_2 \leq d^{n_0}$, $A_3=e^{2\pi i k_2 / d^{n_0}}A_4$. Let \[\tilde{\Gamma}_1=\tilde{\Gamma} \circ \gamma^{-1}_{H,\frac{k_2}{d^{n_0}}},\] then, $\pi_F\circ \tilde{\Gamma}_1=\pi_F \circ \tilde{\Gamma} \circ \gamma^{-1}_{H,\frac{k_2}{d^{n_0}}}= \Gamma \circ \pi_H \circ \gamma^{-1}_{H,\frac{k_2}{d^{n_0}}}= \Gamma \circ \pi_H \implies \tilde{\Gamma}_1$ is a lift of $\Gamma$, hence is of the form, \[\tilde{\Gamma}_1(\lambda,s,t)=(A_1(\lambda),\beta' s +\gamma'(\lambda,t),A_4 t).\]
Without loss of generality, we write $\tilde{\Gamma}$ in place of $\tilde{\Gamma}_1$, and $A_3$ in place of $A_4$. We have $A_3^{d^n(d+1)}=\beta'$, and $A_3^{d^n(d+1)(d-1)}=1$ for $n\geq 0$. For $0\leq l \leq n-1$, there are no terms of degree $d^l(d+1)$  in $\Delta_H(\lambda,t)-\Delta_F(\lambda,t)$. The coefficient of $d^{n-1}(d-1)$ is \[\left(\frac{d}{\delta}\right)^n(1-e^{2\pi i/d})\left(\beta'D_{H,\lambda,1} - A_3^{d^{n-1}(d-1)}D_{F,A_1(\lambda),1}\right).\] Again, since $\Delta_H(\lambda,t)-\Delta_F(\lambda,t)$ is bounded for $n\geq 1$, so \[\beta'D_{H,\lambda,1} - A_3^{d^{n-1}(d-1)}D_{F,A_1(\lambda),1} \to 0,\] as $n\to \infty$. Since $A_3^{d^2}=A_3$, if n is even then $A_3^{d^{n-1}}=A_3^d$ and if n is odd then $A_3^{d^{n-1}}=A_3.$ Hence, 
\[
A_3^{2}D_{H,\lambda,1} =D_{F,A_1(\lambda),1} \text{ and } A_3^{2 d}D_{H,\lambda,1} =D_{F,A_1(\lambda),1}.
\]
Similarly, for $1 \leq i \leq d-1$, 
\[
A_3^{i+1}D_{H,\lambda,i} =D_{F,A_1(\lambda),i} \text{ and } A_3^{(i+1) d}D_{H,\lambda,i} =D_{F,A_1(\lambda),i}.
\] 
In further computations, we use
\begin{equation}\label{DHF}
    A_3^{(i+1) d}D_{H,\lambda,i} =D_{F,A_1(\lambda),i}.
\end{equation}
Thus, $\Delta_H(\lambda,t)-\Delta_F(\lambda,t)= 0$ which implies $\gamma'(\lambda,t)=\gamma'(\lambda,e^{2\pi i/d^n}t)$. Thus, $\gamma'(\lambda,t)=\gamma'(\lambda)$ is independent of $t$.

\medskip
\noindent
{\it Step 4.}
Now, for $D_{H,\lambda,i},D_{F,\lambda,i}\in \mathbb{C}$,
\[
y_H(\lambda,t)=t+\frac{D_{H,\lambda,1}}{t}+\frac{D_{H,\lambda,2}}{t^2}+\cdots+\frac{D_{H,\lambda,d-1}}{t^{d-1}}+O(t^{-d})
\]
and
\[
y_F(\lambda,t)=t+\frac{D_{F,\lambda,1}}{t}+\frac{D_{F,\lambda,2}}{t^2}+\cdots+\frac{D_{F,\lambda,d-1}}{t^{d-1}}+O(t^{-d}).
\]
Moreover
\[
t_H(\lambda,y)=y+\frac{L_{H,\lambda,1}}{y}+\frac{L_{H,\lambda,2}}{y^2}+\cdots+\frac{L_{H,\lambda,d-1}}{y^{d-1}}+O(y^{-d})
\]
and 
\[
t_F(\lambda,y)=y+\frac{L_{F,\lambda,1}}{y}+\frac{L_{F,\lambda,2}}{y^2}+\cdots+\frac{L_{F,\lambda,d-1}}{y^{d-1}}+O(y^{-d}).
\]
\noindent
Since $t_H(\lambda,y_H(\lambda,t))=t$, and $t_F(\lambda,y_F(\lambda,t))=t$,
\begin{align*}
t  = & t+\left[\frac{D_{H,\lambda,1}}{t}+\frac{D_{H,\lambda,2}}{t^2}+\cdots+\frac{D_{H,\lambda,d-1}}{t^{d-1}}+O(t^{-d})\right] \\
& +L_{H,\lambda,1}\left[t+\frac{D_{H,\lambda,1}}{t}+\frac{D_{H,\lambda,2}}{t^2}+\cdots+\frac{D_{H,\lambda,d-1}}{t^{d-1}}+O(t^{-d})\right]^{-1}\\
& +L_{H,\lambda,2}\left[t+\frac{D_{H,\lambda,1}}{t}+\frac{D_{H,\lambda,2}}{t^2}+\cdots+\frac{D_{H,\lambda,d-1}}{t^{d-1}}+O(t^{-d})\right]^{-2}+\cdots, 
\end{align*}
which implies
\begin{equation}{\label{t_H}}
\begin{split}
0  = & \left[\frac{D_{H,\lambda,1}}{t}+\frac{D_{H,\lambda,2}}{t^2}+\cdots+\frac{D_{H,\lambda,d-1}}{t^{d-1}}+O(t^{-d})\right] \\
& +\frac{L_{H,\lambda,1}}{t}\left[1+\frac{D_{H,\lambda,1}}{t^2}+\frac{D_{H,\lambda,2}}{t^3}+\cdots+\frac{D_{H,\lambda,d-1}}{t^{d}}+O(t^{-(d+1)})\right]^{-1}\\
& +\frac{L_{H,\lambda,2}}{t^2}\left[1+\frac{D_{H,\lambda,1}}{t^2}+\frac{D_{H,\lambda,2}}{t^3}+\cdots+\frac{D_{H,\lambda,d-1}}{t^{d}}+O(t^{-(d+1)})\right]^{-2}+\cdots .
\end{split}
\end{equation}
Similarly, 
\begin{equation}{\label{t_F}}
\begin{split}
0 & = \left[\frac{D_{F,\lambda,1}}{t}+\frac{D_{F,\lambda,2}}{t^2}+\cdots+\frac{D_{F,\lambda,d-1}}{t^{d-1}}+O(t^{-d})\right] \\
& +\frac{L_{F,\lambda,1}}{t}\left[1+\frac{D_{F,\lambda,1}}{t^2}+\frac{D_{F,\lambda,2}}{t^3}+\cdots+\frac{D_{F,\lambda,d-1}}{t^{d}}+O(t^{-(d+1)})\right]^{-1}\\
& +\frac{L_{F,\lambda,2}}{t^2}\left[1+\frac{D_{F,\lambda,1}}{t^2}+\frac{D_{F,\lambda,2}}{t^3}+\cdots+\frac{D_{F,\lambda,d-1}}{t^{d}}+O(t^{-(d+1)})\right]^{-2}+\cdots .
\end{split}
\end{equation}
We now prove that for $1\leq i\leq d-1$, $A_3^{(i+1)d}L_{H,\lambda,i}=L_{F,A_1(\lambda),i}$. For simplicity, write
\[
D_H(t)=\sum_{i\ge1}\frac{D_{H,\lambda,i}}{t^i}, \text{ and } D_F(t)=\sum_{i\ge1}\frac{D_{F,\lambda,i}}{t^i}.
\]
Then \eqref{t_H} and \eqref{t_F} give
\begin{equation}\label{eq:H}
0=D_H(t)+\sum_{m\ge1}\frac{L_{H,\lambda,m}}{t^m}\bigl(1+t^{-1}D_H(t)\bigr)^{-m},
\end{equation}
and
\begin{equation}\label{eq:F}
0=D_F(t)+\sum_{m\ge1}\frac{L_{F,\lambda,m}}{t^m}\bigl(1+t^{-1}D_F(t)\bigr)^{-m}.
\end{equation}
Observe that $t^{-1}D_H(t)=O(t^{-2}),$ so that $
\bigl(1+t^{-1}D_H(t)\bigr)^{-m}=1+O(t^{-2}).$
Hence the coefficient of $t^{-i}$ in \eqref{eq:H} is
\[
D_{H,\lambda,i}+L_{H,\lambda,i}+P_{H,i}(
D_{H,\lambda,1},\ldots,D_{H,\lambda,i-1},L_{H,\lambda,1},\ldots,L_{H,\lambda,i-1}),
\]
where $P_{H,i}$ is a polynomial depending only on coefficients of
lower order. Therefore
\[
L_{H,\lambda,i}=-D_{H,\lambda,i}-P_{H,i}(
D_{H,\lambda,1},\ldots,D_{H,\lambda,i-1},L_{H,\lambda,1},\ldots,L_{H,\lambda,i-1}),
\]
so the coefficients $L_{H,\lambda,i}$ are determined recursively by the coefficients $D_{H,\lambda,i}$ and the previously determined $L_{H,\lambda,j}$, $j<i$. Similar recursion is obtained from \eqref{eq:F}. We now prove
the desired result by induction on $i$. For $i=1$, the coefficient of $t^{-1}$ gives
\[
0=D_{H,\lambda,1}+L_{H,\lambda,1}, \text{ and } 0=D_{F,\lambda,1}+L_{F,\lambda,1},
\]
so
\[
L_{F,\lambda,1}=-D_{F,\lambda,1}=-A_3^{2d}D_{H,\lambda,1}=A_3^{2d}L_{H,\lambda,1}.
\]
Assume that
\[
L_{F,\lambda,j}=A_3^{(j+1)d}L_{H,\lambda,j}, \text{   }j<i.
\]
Since
\[
D_{F,\lambda,j}
=
A_3^{(j+1)d}D_{H,\lambda,j},
\]
for all $j\le i$, every monomial appearing in the polynomial
$P_{H,i}$ acquires the common factor $A_3^{(i+1)d}$.
Hence
\[
P_{F,i}(
D_{F,\lambda,1},\ldots,D_{F,\lambda,i-1},L_{F,\lambda,1},\ldots,L_{F,\lambda,i-1})=A_3^{(i+1)d}P_{H,i}(
D_{H,\lambda,1},\ldots,D_{H,\lambda,i-1},L_{H,\lambda,1},\ldots,L_{H,\lambda,i-1}).
\]
Therefore
\begin{align*}
L_{F,\lambda,i} & =-D_{F,\lambda,i}-P_{F,i}(
D_{F,\lambda,1},\ldots,D_{F,\lambda,i-1},L_{F,\lambda,1},\ldots,L_{F,\lambda,i-1}) \\
&=-A_3^{(i+1)d}D_{H,\lambda,i}-A_3^{(i+1)d}P_{H,i}(
D_{H,\lambda,1},\ldots,D_{H,\lambda,i-1},L_{H,\lambda,1},\ldots,L_{H,\lambda,i-1}) \\
&=A_3^{(i+1)d}\left(-D_{H,\lambda,i}-P_{H,i}(
D_{H,\lambda,1},\ldots,D_{H,\lambda,i-1},L_{H,\lambda,1},\ldots,L_{H,\lambda,i-1})\right) \\
&=A_3^{(i+1)d}L_{H,\lambda,i},
\end{align*}
completing the induction. Thus, 
\begin{equation}{\label{asymp}}
    t_F(A_1(\lambda),A_3 y)-A_3 t_H(\lambda,y)=O(y^{-d}).
\end{equation}
Also, from the definition of B\"{o}ttcher function, we have
\[
t_H(\lambda,y)  =y\cdot \left(\frac{p_{H,\lambda}(y)}{y^d}\right)^{1/d}\cdots = y\cdot \left(\frac{p_{H,\lambda}(y)}{y^d}\right)^{1/d}\cdot\left[1+O(y^{-d})\right].\] 
Then
\begin{align}\label{c_H}
A_3 t_H(\lambda,y)-O(y^{-d}) = A_3 y \Biggl[ 1+ & \frac{1}{d}\left(\frac{c_{H,\lambda,d-2}}{y^2}+\cdots +\frac{c_{H,\lambda,0}-\delta_H x}{y^d}\right) \Biggr.\\
& \left. +\frac{1}{2 d}\left(\frac{1}{d}-1\right)\left(\frac{c_{H,\lambda,d-2}}{y^2}+\cdots +\frac{c_{H,\lambda,0}-\delta_H x}{y^d}\right)^2+\cdots\right] 
\end{align}
and
\begin{align}{\label{c_F}}
t_F(A_1(\lambda),A_3 y)-O(y^{-d}) = A_3 y \Biggl[ 1+ & \frac{1}{d}\left(\frac{c_{F,A_1(\lambda),d-2}}{(A_3 y)^2}+\cdots +\frac{c_{F,A_1(\lambda),0}-\delta_F x}{(A_3 y)^d}\right) \Biggr.\\
& \left. +\frac{1}{2 d}\left(\frac{1}{d}-1\right)\left(\frac{c_{F,A_1(\lambda),d-2}}{(A_3 y)^2}+\cdots +\frac{c_{F,A_1(\lambda),0}-\delta_F x}{(A_3 y)^d}\right)^2 +\cdots\right].
\end{align}
Comparing coefficient of $1/y^2$ using \eqref{c_H}, \eqref{c_F} and \eqref{asymp} we get
\[
c_{H,\lambda,d-2}=\frac{c_{F,A_1(\lambda),d-2}}{A_3^2}.
\]
Let, for some $2\leq k\leq d-1$ and for every $2\le i \leq k$
\[c_{H,\lambda,d-i}=\frac{c_{F,A_1(\lambda),d-i}}{A_3^{i}}.\]
Since $A_3\frac{c_{H,\lambda,d-(k+1)}}{d}+A_3\kappa(c_{H,\lambda,d-2},\cdots,c_{H,\lambda,d-k})=\frac{c_{F,A_1(\lambda),d-(k+1)}}{A_3^k d}+A_3\kappa\left(\frac{c_{F,A_1(\lambda),d-2}}{A_3^2},\cdots,
\frac{c_{F,A_1(\lambda)},d-k}{A_3^k}\right)$,
\[
c_{H,\lambda,d-(k+1)}=\frac{c_{F,A_1(\lambda),d-(k+1)}}{A_3^{(k+1)}}.
\]
Thus, for every $2\leq k\leq d$, 
\[
c_{H,\lambda,d-k}=\frac{c_{F,A_1(\lambda),d-k}}{A_3^{k}} \implies A_3^{k} c_{H,\lambda,d-k}=c_{F,A_1(\lambda),d-k}.
\]
We now observe the relationship between the polynomials $q_{F,A_1(\lambda)}$ and $q_{H,\lambda}$,
\begin{align*}
q_{F,A_1(\lambda)}(y) & = y^d+c_{F,A_1(\lambda),d-2}y^{d-2}+\cdots+c_{F,A_1(\lambda),0} \\
&= A_3^d\left(\frac{y}{A_3}\right)^d+A_3^dc_{H,\lambda,d-2}\left(\frac{y}{A_3}\right)^{d-2}+\cdots+A_3^d c_{H,\lambda,0}\\
& = A_3^d q_{H,\lambda}\left(\frac{y}{A_3}\right)
\end{align*}
For $1\leq i\leq 3$, consider the automorphisms $\tilde L_1:\mathbb{D}\times\mathbb{C}^2\to \mathbb{D}\times\mathbb{C}^2$ defined by $\tilde L_1(\lambda,x,y)=(\lambda,x,A_3^{-1} y)$, $\tilde L_2(\lambda,x,y)=(A_1(\lambda),A_3 x,A_3^d y)$ and $\tilde L_3(\lambda,x,y)=(A_1(\lambda),A_3^d x ,y)$. Then,
\[
F\circ \tilde L_3 = \tilde L_2 \circ H \circ \tilde L_1.
\]

\medskip
\no
{\it Step 5.}
Now we study the general case where $\delta_H\neq \delta_F$. In above step, we observed that $\delta_H^{d-1}=\delta_F^{d-1}$. And $H$ lifts to $\tilde{H}$ which is defines by 
\[
\tilde{H}(\lambda,s,t)=\left(\lambda,\frac{\delta_H}{d}s+Q_{H,\lambda}(t),t^d\right).
\]
Let $H_1=H^{d-1}$, then $\deg(H_1)=d_1=d^{d-1}$. Lift of $H_1|_{U_H^+}$ is given by 
\begin{align*}
\tilde{H}^{d-1}(\lambda,s,t)& =  \left(\lambda,\left(\frac{\delta_H}{d}\right)^{d-1}s+\left(\frac{\delta_H}{d}\right)^{d-2}Q_{H,\lambda}(t)+\cdots+Q_{H,\lambda}(t^{d^{d-2}}),t^{d^{d-1}}\right)\\
& =\left(\lambda,\left(\frac{\delta_H}{d}\right)^{d-1}s+\tilde{Q}_{H,\lambda}(t),t^{d^{d-1}}\right),
\end{align*}
where $\tilde{Q}_{H,\lambda}$ is polynomial of degree $d^{d-2}(d+1)$.
The fundamental group of $U_{H_1}^+$ is isomorphic to $\mathbb{Z}[\frac{1}{d_1}]=\mathbb{Z}[\frac{1}{d}]$. For every element $\frac{k}{d_1^n}\in \mathbb{Z}[\frac{1}{d_1}]/\mathbb{Z}$, there exists a unique deck transformation $\gamma_{H_1,\frac{k}{d_1^n}}:\mathbb{D}\times\mathbb{C}\times\mathbb{C}\backslash\mathbb{\overline{D}}\to \mathbb{D}\times\mathbb{C}\times\mathbb{C}\backslash\mathbb{\overline{D}}$ defined by 
\[
(\lambda,s,t)\mapsto\left(\lambda,s+\frac{d_1}{\delta_1}\sum_{l=0}^{\infty}\left(\frac{d_1}{\delta_1}\right)^l\left[\tilde{Q}_{H_1,\lambda}(t^{d_1^l})-\tilde{Q}_{H_1,\lambda}((e^{2\pi i k /d_1^n}t)^{d_1^l})\right],e^{2\pi i k/d_1^n}t\right).
\]
Similarly, let $F_1=F^{d-1}$ is of degree $d_1$. Let $\Gamma:U_{H_1}^+\to U_{F_1}^+$ be the biholomorphism. Note that $\delta_{H_1}=\delta_H^{d-1}=\delta_F^{d-1}=\delta_{F_1}$. Caring out similar set of arguments as in Step 3, we get that its lift $\tilde{\Gamma}:\mathbb{D}\times\mathbb{C}\times\mathbb{C}\backslash\mathbb{\overline{D}}\to \mathbb{D}\times\mathbb{C}\times\mathbb{C}\backslash\mathbb{\overline{D}}$ is of the form \[\tilde{\Gamma}(\lambda,s,t)=(A_1(\lambda),\beta' s+\gamma'(\lambda),A_3 t),\] where $\beta'^{d-1}=1$, $A_3^{d^{d-2}(d+1)d^n}=\beta'$ for every $n \geq 1$. Let $\tilde{A}_3=A_3^{d^{d-2}}$, then for $1\leq i \leq d-1$, \[\tilde{A}_3^{i+1}D_{H,\lambda,i}=D_{F,A_1(\lambda),i},\] which gives $q_{F,A_1(\lambda)}(y)=\tilde{A}_3^d q_{H,\lambda}(\tilde{A}_3^{-1} y)$. Further $\delta_F=\alpha_1\delta_H$ for some $\alpha_1^{d-1}=1.$ Consider the automorphisms $\tilde L_4(\lambda,x,y)=(A_1(\lambda),\tilde{A}_3 x,\tilde{A}_3^d y)$ and $\tilde{L}_5(\lambda,x,y)=(A_1^{-1}(\lambda),\tilde{A}_3^{-d}\alpha_1 x, \tilde{A}_3^{-1}y)$ on $\mathbb{D}\times\mathbb{C}^2$. Then, $F=\tilde{L}_4\circ H\circ \tilde{L}_5$.

\section{Non-Compact Parameter Space}
Recall that $H:\mathbb{C}\times \mathbb{C}^2 \to \mathbb{C}\times \mathbb{C}^2$ is defined in \eqref{skewhenon} by 
\begin{equation*}
     H(\lambda,x,y)=(c\lambda,H_\lambda(x,y)).
\end{equation*}
For $n \geq 1$, let
\[
x_n^\lambda = proj_1 \circ H_\lambda^n, \text{\hspace{1mm} and } y_n^\lambda= proj_2 \circ H_\lambda^n.
\]
We first consider the case $|c|>1$. Recall, for $R>0$, 
\begin{center}
$V_R^+ = \{(\lambda,x,y) \in \mathbb{C}^3 : |y| > \max \{R, |x|, |\lambda|^{\tilde{d}+1} \}\},$\\
\vspace{2mm}
$V_R^-= \{(\lambda,x,y) \in \mathbb{C}^3 : |x| > \max \{R, |y|\}, |\lambda|<1 \}.$
\end{center}
\noindent
By \cite[Section 3]{BP}, for sufficiently large $R$, $H(V_R^+)\subset V_R^+$, and $H^{-1}(V_R^-)\subset V_R^-$ and 
\[
U_H^+ = \bigcup_{n\geq 0} H^{-n}(V^+_R).
\]

\medskip
\noindent
Now we prove Theorem \ref{thm3}.\\
{\it Step 1.}
Since, for $0\leq j\leq d-1$, $c_j$'s are polynomials in $\lambda$ of degree at most $\tilde{d}$, there exists $M>1$ such that 
\[
|c_j(\lambda)|<M|\lambda|^{\tilde{d}}_{+},
\]
where $|\lambda|_+=\max\{1,|\lambda|\}$. Thus, for $(\lambda,x,y)\in V_R^+$,
\[
|c_j(\lambda)|<M|\lambda|^{\tilde{d}}_{+} < M|y|^{\tilde{d}/(\tilde{d}+1)} \implies \frac{|c_j(\lambda)|}{|y|}<\frac{M}{|y|^{1/(\tilde{d}+1)}}.
\]
Choose sufficiently large $R$ such that
\[
\left|\frac{p_\lambda(y)-\delta x}{y^d} -c_H\right| = \left|\frac{c_{\lambda,d-1}}{y}+ \cdots +\frac{c_{\lambda,0}-\delta x}{y^d}\right| <\left|\frac{c_H}{2}\right|.
\]
for $(x,y)\in V_R^+$.
Thus, we can choose a branch of the logarithm, denoted by $\beta_H$, on $V_R^+$ such that
\[
e^{\beta_H(\lambda,x,y)} = \frac{p_\lambda(y)-\delta x}{y^d} = c_H +\frac{c_{\lambda,d-1}}{y}+ \cdots +\frac{c_{\lambda,0}-\delta x}{y^d}.
\]
Inductively, for any natural $n\geq 1$, 
\[
proj_2 \circ H_\lambda^n(x,y)=y^{d^n} e^{\beta_H(H^{n-1}(\lambda,x,y))+ \cdots +d^{n-1}\beta_H(\lambda,x,y)}.
\]
Let 
\begin{equation}\label{eq5.1}
\begin{split}
\phi_H^+(\lambda,x,y) & = \lim_{n \to \infty} (proj_2 \circ H_\lambda^n(x,y))^{1/d^n}\\
& = \lim_{n \to \infty} y e^{\frac{1}{d}\beta_H(\lambda,x,y)+\cdots +\frac{1}{d^n}\beta_H(H^{n-1}(\lambda,x,y))}\\
& =  ye^{\gamma_H^+(\lambda,x,y)}.
\end{split}
\end{equation}
Since for  large $R>0$, $\beta_H$ is bounded on $V_R^+$, the convergence is uniform on compact subsets of $V_R^+$. Therefore, $\phi_H^+:V_R^+ \to \mathbb{C}\backslash \mathbb{\overline{D}}$ is holomorphic. Moreover, 
\begin{equation}\label{eq5.2}
\begin{split}
    \phi_H^+(H(\lambda,x,y)) & = \lim_{n \to \infty} proj_2\circ H_\lambda(x,y) e^{\frac{1}{d}\beta_H(H(\lambda,x,y))+\cdots +\frac{1}{d^n}\beta_H(H^{n}(\lambda,x,y))}\\
    &= \lim_{n \to \infty} y^d e^{\beta_H(\lambda,x,y)} e^{\frac{1}{d}\beta_H(H(\lambda,x,y))+\cdots +\frac{1}{d^n}\beta_H(H^{n}(\lambda,x,y))}\\
    &=(\phi_H^+(\lambda,x,y))^d.
    \end{split}
\end{equation}
As $y \to \infty$ in $V_R^+$, $\phi_H^+ \sim c_H^{\frac{1}{d-1}}y$.

\medskip
\no
{\it Step 2.}
Define the closed holomorphic 1-form on $V_R^+$ as $\omega_H^+ =\mathrm{d}\phi_H^+/\phi_H^+$ and extend it to $U_H^+$ using the functional equation \eqref{eq5.2}. For $n\geq 1$, on $H^{-n}(V_R^+)$ we define 
\[
\omega_H^+ = \frac{1}{d^n}(H^n)^*\omega_H^+.
\]
We define
\[
F_H^+ : \pi_1(U_H^+) \to \mathbb{Z}\left[\frac{1}{d}\right] \text{\hspace{1mm} given by } [C] \mapsto \frac{1}{2\pi i}\int_C \omega_H^+.
\]
Similar to compact base space case, since $\omega_H^+= \mathrm{d}y/y +\mathrm{d}\gamma_H^+$ and $\gamma_H^+$ is holomorphic on $V_R^+$, $F_H^+(\pi_1(V_R^+))=\mathbb{Z}$. Using \eqref{eq5.2}, we have $H^*\omega_H^+=d\omega_H^+$. Therefore, since $H:U_H^+\to U_H^+$ is a homeomorphism, the same argument as in compact case shows that $F_H^+$ is a group isomorphism.
Fix $z_0 \in V_R^+$. Let $A_H=\{(z,C): z \in U_H^+$, and $C$ is a path from $z_0$ to $z$ in $U^+_H\}$. Define an equivalence relation $\sim_H^+$ on $A_H$
as 
\begin{center}
$(z,C) \sim_H^+ (z',C')$ if and only if $z=z'$ and $[C\overline{C'}] \in (F_H^+)^{-1}(\mathbb{Z})$.
\end{center}
Let $\hat{U}_H^+$ be the set of equivalence classes $A_H/ \sim_H^+$ of this relation. Then the map $\pi_H:\hat{U}_H^+ \to U^+_H$ defined by $[(z,C)] \mapsto z$ is a covering map. Using $\pi_H$, we pull back the analytic structure of $U^+_H$ to give an analytic structure to $\hat{U}_H^+$ such that $\pi_H$ is holomorphic. Also, $\hat{U}_H^+$ has an open subset $\hat{V}_R^+=\{[(z,C)] \in \hat{U}_H^+: \pi_H([(z,C)]) \in V_R^+\}$ biholomorphic to $V_R^+$.
Define $\hat{\phi}_H^+ :\hat{U}_H^+ \to \mathbb{C}\backslash \mathbb{\overline{D}}$ by 
\begin{equation}\label{eq5.4}
[(z,C)] \mapsto \phi_H^+(z_0)e^{\int_C \omega_H^+}
\end{equation}
Note that, if $[(z,C)] \in \hat{V}_R^+$, then $\hat{\phi}_H^+([(z,C)])=\phi_H^+(z).$
Fix a path $L_H$ from $z_0$ to $H(z_0)$ in $V^+_R$, then the map $\hat{H}:\hat{U}_H^+ \to \hat{U}_H^+$ defined by \begin{equation}\label{eq5.5}
    [(z,C)] \mapsto [(H(z),L_H H(C))]
\end{equation}
is a lift of $H$, i.e., $H \circ \pi_H = \pi_H \circ \hat{H}.$ Note that every point has exactly $d$ preimages under $\hat{H}$. Since we have $H(V_R^+)\subset V_R^+$, so for $[(z,C)]\in \hat{V}_R^+$, $[(H(z),L_H H(C))] \in \hat{V}_R^+$. Hence, we obtain the increasing sequence of open sets
\[
\hat{V}_R^+ \subset \hat{H}^{-1}(\hat{V}_R^+) \subset \hat{H}^{-2}(\hat{V}_R^+)\subset \cdots.
\]
Moreover, if $[(z,C)]\in \hat{U}^+_H$, then there exist $n\geq 1$ such that 
\begin{center}
$H^n(z) \in V_R^+$ and $L_H H(L_H)\cdots H^{n-1}(L_H) H^n(C) \subset V^+_R$,
\end{center}
hence 
\[
\hat{U}_H^+ = \bigcup_{n \geq 0 } \hat{H}^{-n}(\hat{V}_R^+).
\]
Let $[(z,C)]\in \hat{U}_H^+$, then 
\begin{equation}\label{eq5.6}
\hat{\phi}_H^+(\hat{H}([(z,C)])) = \phi_H^+(z_0)e^{\int_{L_H}\omega_H^+}e^{\int_{H(C)}\omega_H^+} = \phi_H^+(H(z_0))e^{d\int_C\omega_H^+}= (\hat{\phi}_H^+([(z,C)]))^d.
\end{equation}

\medskip
\noindent
{\it Step 3.}
Let $M_1>0$ be a sufficiently large constant. Let
\[
U_R^+=\{(\lambda,x,y) \in V_R^+ : |\phi_H^+(\lambda,x,y)|> M_1 \max\{R,|x|, |\lambda|^{\tilde{d}+1}\}.
\]
Note that there exists positive constants $m_1$ and $m_2$ such that 
\[
m_1\leq \left|\frac{\phi_H^+(\lambda,x,y)}{y}\right|\leq m_2
\]
for $(\lambda,x,y)\in V_R^+$. Therefore, there exists $m_0 \in \mathbb{N}$, such that $H^{m_0}(V^+_R) \subset U^+_R$. Equation \eqref{eq5.2} implies that $H(U_R^+)\subset U_R^+$. Thus, we have the increasing sequence of open sets
\begin{center}
$U_R^+ \subset H^{-1}(U_R^+) \subset (H)^{-2}(U_R^+) \subset \cdots$.
\end{center}
Consider the biholomorphic coordinate change
\[
(\lambda,x,y) \mapsto (\lambda,x,\phi_H^+(\lambda,x,y))=(\lambda,x,t),
\]
which maps $U_R^+$ onto $\{(\lambda,x,t) \in  \mathbb{C}^3 : |t|>M_1 \max\{R,|x|, |\lambda|^{\tilde{d}+1}\} \}$ with inverse 
\[
(\lambda,x,t) \mapsto (\lambda,x,\theta_H^+(\lambda,x,t)).
\]
Also, consider the map $(\lambda,x,t) \mapsto (\lambda,\tilde{s},t)$, where
\[
\tilde{s}=t \int_0^x \frac{\partial \theta_H^+}{\partial t}(\lambda,\xi,t)d\xi.
\]
In the coordinates $(\lambda,\tilde{s},t)$, the map $H$ is given by
\[
(\lambda,\tilde{s},t)\mapsto (\lambda,x,y) \mapsto (c\lambda,y,p_\lambda(y)-\delta x) \mapsto (c\lambda,\tilde{s}_1,t_1),
\]
where $t_1=t^d$. Comparing the Jacobian determinant of the composition with the product of the Jacobian determinants of the individual coordinate changes, we obtain 
\[
dc\frac{\partial \tilde{s}_1}{\partial \tilde{s}} t^{d-1}=\delta c t^{d-1}.
\] 
Thus, $\tilde{s}_1 = \frac{\delta}{d} \tilde{s} +C(\lambda,t)$, where $C(\lambda,t)$ is holomorphic on $\mathbb{C}\times\{t\in \mathbb{C}:|t|>M_1R\}$.
Computing the image of $(\lambda,0,t)$ under the map $(\lambda,\tilde{s},t) \mapsto (c\lambda,\tilde{s}_1,t_1)$, we get
\[
C_H(\lambda,t)=  t^d \int_0^{\theta_H^+(\lambda,0,t)} \frac{\partial \theta_H^+}{\partial t}(c\lambda,\xi,t^d)d\xi= l_{d+1}t^{d+1}+l_d(\lambda)t^d+ \cdots +\frac{l_{-1}(\lambda)}{t}+\cdots.
\]
The coefficients in the Laurent expansion are holomorphic for $\lambda \in \mathbb{C}$, and the degree $d+1$ coefficient is constant. For each $\lambda \in \mathbb{C}$, let $Q_H(\lambda,t)$ be the polynomial part of $C_H(\lambda,t)$, and $C_H^-(\lambda,t)$ is the singular part. Let
\[
R_H(\lambda,t)=\frac{d}{\delta}C_H^-(\lambda,t)+\left(\frac{d}{\delta}\right)^2 C_H^-(c\lambda,t^d)+\left(\frac{d}{\delta}\right)^3 C_H^-(c^2\lambda,t^{d^2})+\cdots.
\] 
Define $s=\tilde{s}+R_H(\lambda,t)$. Since $C_H(\lambda,t)+R_H(c\lambda,t^d)-\frac{\delta}{d}R_H(\lambda,t)=Q_H(\lambda,t)$, $H$ in $(\lambda,s,t)$ coordinates becomes 
\[
(\lambda,s,t) \mapsto \left(c\lambda, \frac{\delta}{d} s +Q_H(\lambda,t),t^d\right).
\]  
Let $\psi_H^+=s$, then it satisfies, 
\begin{equation}\label{eq5.7}
\psi_H^+(H(\lambda,x,y))=\frac{\delta}{d} \psi_H^+(\lambda,x,y) +Q_H(\lambda,\phi_H^+(\lambda,x,y)).
\end{equation} 
For large $R$, the map $\Phi_H^+:(\lambda,x,y) \mapsto(c\lambda,\psi_H^+(\lambda,x,y),\phi_H^+(\lambda,x,y))$ is injective and holomorphic map on $U_R^+$, and the set $\{(\lambda,s,t): M|\lambda|^{\tilde{d}+1}<|t|,|s|<M_2|t|^2 ,|t|> M\} $ lies in the image of the map, where $M$(lagre), $M_2$(small) are constants.

\medskip
\noindent
{\it Step 4.}
Let $\hat{U}_R^+=\{[(z,C)] \in \hat{V}_R^+ : \pi_H([(z,C)]) \in U^+_R\}$, there exists $n\geq 1$ such that $\hat{H}^n(\hat{V}_R^+) \subset \hat{U}_R^+$, and
\begin{equation}\label{eq5.8}
\hat{U}_H^+ = \bigcup_{n \geq 0 } \hat{H}^{-n}(\hat{U}_R^+).
\end{equation}
For $[(z,C)] \in \hat{U}_R^+$, define 
\[
\hat{\psi}_H^+([(z,C)])=\psi_H^+(z),
\]
and extend it to $\hat{U}^+_H$ by using the functional equation \eqref{eq5.7} satisfied by $\psi_H^+$, so that it satisfies
\[
\hat{\psi}_H^+ \circ \hat{H}([(z,C)])= \frac{\delta}{d}\hat{\psi}_H^+([(z,C)])+ Q_H(\lambda,\hat{\phi}_H^+([(z,C)])),
\]
for $[(z,C)] \in \hat{U}_H^+$ with $z=(\lambda,x,y).$
Define $\hat{\Phi}_H^+: \hat{U}_H^+ \to \mathbb{C}^2 \times \mathbb{C}\backslash\mathbb{\overline{D}}$ by
\[
\hat{\Phi}_H^+([(z,C)])=\left(\lambda,\hat{\psi}_H^+([(z,C)]),\hat{\phi}_H^+([(z,C)])\right)
\]
for $[(z,C)] \in\hat{U}_H^+$ with $z=(\lambda,x,y).$ Note that the map is injective on $\hat{U}_R^+$.
Let $G_H:\mathbb{C}^2 \times \mathbb{C}\backslash\mathbb{\overline{D}} \to \mathbb{C}^2 \times \mathbb{C}\backslash\mathbb{\overline{D}}$ be defined by
\[
G_H(\lambda,s,t)=\left(c\lambda,\frac{\delta}{d}s+Q_H(\lambda,t),t^d\right).
\]
Then for any $n \geq 1$, 
\begin{equation}\label{eq5.9}
\hat{\Phi}_H^+ \circ \hat{H}^n=G_H^n \circ \hat{\Phi}_H^+
\end{equation}
Let $W_H= \hat{\Phi}_H^+(\hat{U}_R^+)$. Since $\hat{\Phi}_H^+ \circ \hat{H}=G_H \circ \hat{\Phi}_H^+$ and $\hat{H}(\hat{U}_R^+) \subset \hat{U}_R^+$, it follows that $G_H(W_H)\subset W_H$. Therefore, we have the increasing sequence of sets
\[
W_H \subset G_H^{-1}(W_H)\subset (G_H^2)^{-1}(W_H) \cdots.
\]
Furthermore, since $\{(\lambda,s,t): M|\lambda|^{\tilde{d}+1}<|t|,|s|<M_2|t|^2 ,|t|> M\} \subset \Phi_H^+(U_H^+)$, so $\mathbb{C}^2\times \mathbb{C}\backslash\mathbb{\overline{D}} \subset \bigcup_{n \geq 0} (G_H^n)^{-1}(W_H)$ and the other inclusion is trivial. Therefore, 
\[
\mathbb{C}^2\times \mathbb{C}\backslash\mathbb{\overline{D}} = \bigcup_{n \geq 0} (G_H^n)^{-1}(W_H).
\]
Note that \eqref{eq5.9} gives $\hat{\Phi}_H^+((\hat{H}^n)^{-1}(\hat{U}_R^+)) \subset (G_H^n)^{-1}(W_H)$. We now prove that $\hat{\Phi}_H^+: \hat{U}_H^+ \to \mathbb{C}^2 \times \mathbb{C}\backslash\mathbb{\overline{D}}$ is a biholomorphism by showing that for each $n \geq 1$, the map 
\[
\hat{\Phi}_H^+ |_{(\hat{H}^n)^{-1}(\hat{U}_R^+)} :(\hat{H}^n)^{-1}(\hat{U}_R^+) \to (G_H^n)^{-1}(W_H)
\]
is a bijection.
For injectivity, let $\hat \Phi_H^+([(z_1,C_1)])=\hat \Phi_H^+([(z_2,C_2)])$. Then \eqref{eq5.9} gives 
\[
\hat{\Phi}_H^+ \circ \hat{H}^n([(z_1,C_1)])=\hat{\Phi}_H^+ \circ \hat{H}^n([(z_2,C_2)])
\]
for every $n\geq 1$. We choose $n\geq 1$ large such that $\hat H^n([(z_1,C_1)]),\hat H^n([(z_2,C_2)])\in \hat U_R^+$. Since $\hat{\Phi}_H^+$ is injective on $\hat{U}_R^+$, $\hat H^n([(z_1,C_1)])=\hat H^n([(z_2,C_2)])$ which implies $H^n(z_1)=H^n(z_2)$. Further $H$ is injective, thus, $z_1=z_2$. Observe that if $[(z,C_1)],$ $[(z,C_2)]\in \hat{U}_H^+$ be such that $\hat{\phi}_H^+([(z,C_1)])=\hat{\phi}_H^+([(z,C_2)])$, then $[(z,C_1)]=[(z,C_2)].$ Thus, $[z_1,C_1]=[z_2,C_2]$. Moreover, if $[(z,C)] \in \hat{U}_R^+$, then $\{\hat{H}^{-n}([(z,C)])\}$ has $d^n$ elements which by \eqref{eq5.9} under $\hat{\Phi}^+_H$ are mapped to elements of $\{(G_H^n)^{-1}(\hat{\Phi}^+_H([(z,C)]))\}$ which also has $d^n$ elements. Since $\hat{\Phi}_H^+$ is injective on $(\hat{H}^n)^{-1}(\hat{U}_R^+)$ and both $\{(\hat{H}^n)^{-1}([(z,C)])\}$ and $\{(G_H^n)^{-1}(\hat{\Phi}_H^+([(z,C)]))\}$ have same cardinality, $\hat{\Phi}_H^+|_{(\hat H^n)^{-1}(\hat U_R^+)}$ is surjective.

\medskip
\noindent
Since $H^{-1}(\lambda,x,y)=\left(c^{-1}\lambda,H^{-1}_{c^{-1}\lambda}(x,y)\right)$, the case $|c|<1$ can be treated similarly.

\bibliographystyle{amsplain}

\end{document}